\documentclass[final]{dmtcs-episciences}
\usepackage[utf8]{inputenc}
\usepackage{subfigure}
\usepackage[round]{natbib}

\pdfoutput=1
\usepackage[T1]{fontenc}
\usepackage{amsmath,amssymb,amsbsy}
\usepackage{mathtools}
\usepackage{multirow}
\usepackage{amsfonts}
\usepackage{amsthm}
\usepackage{amssymb}
\usepackage{graphicx, enumerate}
\usepackage{amsmath}
\usepackage{latexsym}
\usepackage{longtable}
\usepackage{tabularx}
\usepackage{amsmath}
\usepackage{amsfonts}
\usepackage{amsthm}
\usepackage{setspace}
\usepackage{graphicx}
\usepackage{float}
\usepackage{rotating}
\usepackage{tikz}
\usepackage{verbatim}
\usepackage{cleveref}
\usepackage{appendix}

\newtheorem{thm}{Theorem}[section]
\newtheorem{lem}[thm]{Lemma}

\newtheorem{prop}[thm]{Proposition}
\newtheorem{cor}[thm]{Corollary}
\newtheorem{prob}[thm]{Problem}
\newtheorem{dfn}[thm]{Definition}

\newtheorem{conj}[thm]{Conjecture}

\newcommand{\zet}{\mathbb{Z}}

\author[Sylwia Cichacz and Karol Suchan]{Sylwia Cichacz\affiliationmark{1}
\and Karol Suchan\affiliationmark{2,1}}

\title[Zero-sum partitions of Abelian groups]{Zero-sum partitions of Abelian groups and their applications to magic- and antimagic-type labelings}

\affiliation{
AGH University, Krakow, Poland\\
Universidad Diego Portales, Santiago, Chile}

\keywords{Abelian group, zero-sum sets, irregular labeling, magic-type labeling, antimagic-type labeling}

\begin{document}

\publicationdata{vol. 26:3}{2024}{14}{10.46298/dmtcs.12361}{2023-10-03; 2023-10-03; 2024-04-26; 2024-09-19}{2024-09-25}

\maketitle

\begin{abstract}
The following problem has been known since the 80s. Let $\Gamma$ be an Abelian group of order $m$ (denoted $|\Gamma|=m$), and let $t$ and $\{m_i\}_{i=1}^{t}$, be positive integers such that $\sum_{i=1}^t m_i=m-1$. Determine when $\Gamma^*=\Gamma\setminus\{0\}$, the set of non-zero elements of $\Gamma$, can be partitioned into disjoint subsets $\{S_i\}_{i=1}^{t}$ such that $|S_i|=m_i$ and $\sum_{s\in S_i}s=0$ for every  $1 \leq i \leq t$. Such a subset partition is called a \textit{zero-sum partition}.

$|I(\Gamma)|\neq 1$, where $I(\Gamma)$ is the set of involutions in $\Gamma$, is a necessary condition for the existence of zero-sum partitions. In this paper, we show that the additional condition of $m_i\geq 4$ for every $1 \leq i \leq t$, is sufficient. Moreover, we present some applications of zero-sum partitions to magic- and antimagic-type labelings of graphs.

\end{abstract}




\section{Introduction}\label{sec:intro}

\subsection{Preliminaries}\label{subsec:preliminaries}

For standard terms and notation in graph theory, the reader is referred to the textbook by~\citet{Diestel} and the monograph by~\citet{Brandstadt}, an introduction to magic- and antimagic-type labelings can be found in the monograph by~\citet{BMRS}, and the fundamental results in abstract algebra that we use can be found in the textbook by~\citet{Gallian}.
 
Let $\Gamma$ be an Abelian group of order $m$ with the operation denoted by $+$.  For convenience, we will denote $\sum_{i=1}^k a$ by $ka$, the inverse of $a$ by $-a$, and $a+(-b)$ by $a - b$. Moreover, we will write $\sum_{a\in S}{a}$ for the sum of all elements in $S$. The identity element of $\Gamma$ will be denoted by $0$, and the set of non-zero elements of $\Gamma$ by $ \Gamma^*$.
Recall that any element $\iota\in\Gamma$ of order 2 (i.e., $\iota\neq 0$ and $2\iota=0$) is called an \emph{involution}. A non-trivial finite group has an involution if and only if the order of the group is even. We will write $I(\Gamma)$ for the set of involutions of $\Gamma$. 

The fundamental theorem of finite Abelian groups states that every finite Abelian group $\Gamma$ is isomorphic to the direct product of some cyclic subgroups of prime-power orders~\citep{Gallian}. In other words, there exists a positive integer $k$, (not necessarily distinct) prime numbers $\{p_i\}_{i=1}^{k}$, and positive integers $\{\alpha_i\}_{i=1}^{k}$, such that
$$\Gamma\cong\zet_{p_1^{\alpha_1}}\times\zet_{p_2^{\alpha_2}}\times\ldots\times\zet_{p_k^{\alpha_k}} \mathrm{, with}\; m = p_1^{\alpha_1}\cdot p_2^{\alpha_2}\cdot\ldots\cdot p_k^{\alpha_k},$$ where $m$ is the order of $\Gamma$. Moreover, this group factorization is unique (up to the order of terms in the direct product). {Since any cyclic finite group of even order has exactly one involution, if $e$ is the number of cyclic subgroups in the factorization of $\Gamma$ whose order is even, then $|I(\Gamma)|=2^e-1$.}

Because the results presented in this paper are invariant under the isomorphism between groups ($\cong$), we only need to consider one group in every isomorphism class. Our presentation will be focused on groups of the form $\zet_{p_1^{\alpha_1}}\times\zet_{p_2^{\alpha_2}}\times\ldots\times\zet_{p_k^{\alpha_k}}$ with prime numbers $\{p_i\}_{i=1}^{k}$ and positive integers $\{\alpha_i\}_{i=1}^{k}$. For an Abelian group $\Gamma \cong U \times H$, for a pair $(u,v)$ with $u \in U$ and $ v\in H$, we will use the notation $(u,v)$ also for the corresponding element of $\Gamma$. 

Recall that, for a prime number $p$, a group the order of which is a power of $p$ is called a \textit{$p$-group}. Given a finite Abelian group $\Gamma$, \textit{Sylow $p$-subgroup} of $\Gamma$ is the maximal subgroup $L$ of $\Gamma$ the order of which is a power of $p$. For example, by the fundamental theorem of finite Abelian groups, it is easy to see that any finite Abelian group $\Gamma$ can be factorized as $\Gamma\cong L\times H$, where $L$ is the Sylow $2$-group of $\Gamma$ and the order of $H$ is odd.

Let us denote the number of involutions in $\Gamma$ by $|I(\Gamma)|$. Since any finite cyclic  group of even order has exactly one involution, if $e$ is the number of cyclic components in the factorization of $\Gamma$ whose order is even, then $|I(\Gamma)|=2^e-1$.

Recall that the sum of all elements of a group $\Gamma$ is equal to the sum of its involutions and the identity element. {The following lemma is well known. The readers can consult \citet{ref_ComNelPal} for a proof.}
\begin{lem}[\citep{ref_ComNelPal}]\label{involutions} Let $\Gamma$  be an Abelian group.
\begin{itemize}
 \item[-] If $|I(\Gamma)|=1$, then $\sum_{g\in \Gamma}g= \iota$, where $\iota$ is the involution.
\item[-] If $|I(\Gamma)|\neq 1$, then $\sum_{g\in \Gamma}g=0$.
\end{itemize}
\end{lem}

 \subsection{Main Problem}\label{subsec:problem}

In 1981 Tannenbaum introduced the following problem of partitioning in Abelian groups.
\begin{prob}[\citep{Tannenbaum1}]\label{problemT}
Let $\Gamma$ be an Abelian group of order $m$. Let $t$ be a positive integer and $\{m_i\}_{i=1}^{t}$ an integer partition of $m-1$. Let $\{w_i\}_{i=1}^{t}$ be arbitrary elements of $\Gamma$ (not necessary distinct). Determine when the elements of $\Gamma^*$ can be partitioned into subsets $\{S_i\}_{i=1}^{t}$ (i.e., subsets $S_i$ are pairwise disjoint and their union is $\Gamma^*$) such that $|S_i|=m_i$ and $\sum_{s\in S_i}s=w_i$ for every  $1\leq i \leq t$.\end{prob}

{If a respective subset partition of $\Gamma^*$ exists, then we say that $\{m_i\}_{i=1}^t$ is \textit{realizable} in $\Gamma^*$ with $\{w_i\}_{i=1}^{t}$. We will not indicate $\Gamma^*$ or $\{w_i\}_{i=1}^{t}$ explicitly when it is clear from the context.} Note that realizability of $\{m_i\}_{i=1}^{t}$ implies that $\sum_{i=1}^t w_i=\sum_{g\in \Gamma}g$. 

Note that we use sequences $\{m_i\}_{i=1}^t$ only for ease of presentation, whereas we interpret them as multisets. In other words, only the values present in the sequence and their multiplicities are relevant for a sequence to be realizable, not their order. 

Since, throughout the paper, we often consider subsets of fixed cardinalities, let us present an abbreviated notation. Given a set, any of its subsets of cardinality $k$ is called a \textit{$k$-subset}.

The case most studied in the literature is the {\em Zero-Sum Partition (ZSP)} problem, i.e., when $w_i=0$ for every $1\leq i \leq t$. (A respective subset partition of $\Gamma^*$ is called a \textit{zero-sum partition} of $\Gamma$.)

In general, in the context of any subset $Z$ of $\Gamma$, we will say that $\{S_i\}_{i=1}^t$, a partition of $Z$ into zero-sum subsets with $|S_i|=m_i$ for every  $1 \leq i \leq t$, realizes $\{m_i\}_{i=1}^t$ in $Z$.

Note that, by Lemma \ref{involutions}, an Abelian group $\Gamma$ with $|I(\Gamma)|=1$ does not admit zero-sum partitions {(for every $t$ and $\{m_i\}_{i=1}^{t}$)}. Moreover, if $\{m_i\}_{i=1}^{t}$ is realizable in $\Gamma^*$, then necessarily $m_i\geq 2$ for every  $1\leq i \leq t$. It was proved that this condition is sufficient if and only if $|I(\Gamma)|=0$~\citep{Tannenbaum1,Zeng} or $|I(\Gamma)|=3$~\citep{Zeng}. I.e., in these cases, for every positive integer $t$ and integer partition $\{m_i\}_{i=1}^{t}$ of $m-1$ with $m_i \geq 2$ for every $1 \leq i \leq t$, $\{m_i\}_{i=1}^{t}$ is realizable in $\Gamma^*$ (with $\{w_i\}_{i=1}^{t}$ where $w_i=0$ for every  $1 \leq i \leq t$). In \citet{CicrSuch}, we generalized this condition with the following definition.

\begin{dfn}Let $\Gamma$ be a finite Abelian group of order $m$. We say that $\Gamma$ has {\em $x$-Zero-Sum Partition Property ($x$-ZSPP)} if, for every positive integer $t$ and integer partition $\{m_i\}_{i=1}^t$ of $m-1$ with $m_i \geq x$ for every $1 \leq i \leq t$, there exists a subset partition $\{S_i\}_{i=1}^t$ of $\Gamma^*$ with $|S_i| = m_i$ and $\sum_{s\in S_i}s = 0$ for every $1 \leq i \leq t$. 
\end{dfn}

For every finite Abelian group $\Gamma$ such that $|I(\Gamma)|=0$ (equivalently, the order of $\Gamma$ is odd), a stronger version of $2$-ZSPP was proved. Let us give a brief explanation.

In 1957, inspired by Steiner triples research, Skolem posed the following question \citep{ref_Sko57}: {For $n\equiv 1\pmod 6$, does there exist a partition of the set of non-zero elements of the cyclic group $\mathbb{Z}_n$ into triples such that the sum of elements in each subset is congruent to $0\pmod n$?

This question received an affirmative answer \citep{ref_Han,ref_Sko57, ref_Sko58} and served as a starting point for a more general problem posed by Tannenbaum. Namely, we call a $6$-subset $C$ of an Abelian group $\Gamma$ \textit{good} if $C = \{c, d,-c -d,-c,-d, c + d\}$ for some $c$ and $d$ in $\Gamma$. Notice that the sum of elements of a good $6$-subset is $0$. Moreover, it can be partitioned into three zero-sum $2$-subsets or two zero-sum $3$-subsets.

 The following definition was given by~\citet{Tannenbaum1}.

\begin{dfn}\label{dfn:skolem}Let $\Gamma$ be a finite Abelian group of order $m=6k+s$ for a non-negative integer $k$ and $s\in\{1,3,5\}$. A partition of $\Gamma^*$ into $k$ good $6$-subsets and $(s-1)/2$ zero-sum $2$-subsets is called a {\em Skolem partition of $\Gamma^*$}.
\end{dfn}

Note that, if $\Gamma$ is an Abelian group of order $m$ such that every integer partition $\{m_i\}_{i=1}^t$ of $m-1$ with $m_i \in \{2, 3\}$ for every $1 \leq i \leq t$ is realizable in $\Gamma^*$, then also every integer partition $\{m'_i\}_{i=1}^{t'}$ of $m-1$ with $m'_i \geq 2$ for every $1 \leq i \leq t'$ is realizable. Similarly, it is easy to see that if a Skolem partition of $\Gamma^*$ exists, then every integer partition $\{m_i\}_{i=1}^t$ of $m-1$ with $m_i \in \{2, 3\}$ for every $1 \leq i \leq t$ is realizable. So the following theorem by Tannenbaun indeed offers a stronger version of $2$-ZSPP.

\begin{thm}[\citep{Tannenbaum1}]\label{Tannenbaum1}Let $\Gamma$ be a finite Abelian group such that $|I(\Gamma)|=0$, then $\Gamma^*$ has a Skolem partition. 
\end{thm}

The following theorem was first conjectured by~\citet{KLR} (they also showed the necessity), and later proved by~\citet{Zeng}.
\begin{thm}[\citep{Zeng}]\label{Zeng}Let $\Gamma$ be a finite Abelian group of order $m$. $\Gamma$ has $2$-ZSPP if and only if $|I(\Gamma)|\in\{0,3\}$. 
\end{thm}

The above theorem confirms for the case of $|I(\Gamma)|=3$ the following conjecture stated by Tannenbaum. 

\begin{conj}[\citep{Tannenbaum1}]\label{conjectureT}Let $\Gamma$ be a finite Abelian group of order $m$ with $|I(\Gamma)|>1$. Let $R =\Gamma^* \setminus I(\Gamma)$. For every positive integer $t$ and integer partition $\{m_i\}_{i=1}^t$ of $m-1$ with {$m_i \geq 2$ for every $1 \leq i \leq |R|/2$ and {$m_i \geq 3$} for every $|R|/2+1\leq i\leq t$}, there is a subset partition $\{S_i\}_{i=1}^t$ of $\Gamma^*$ such that $|S_i| = m_i$ and $\sum_{s\in S_i}s = 0$ for every $1 \leq i \leq t$. 
\end{conj}

It was shown independently by a few authors that the conjecture is also true for $\Gamma\cong (\zet_2)^n$ with $n>1$ (notice that in this case $I(\Gamma)=\Gamma^*$, so $R=\emptyset$). 

\begin{thm}[\citep{ref_CaccJia,Egawa,Tannenbaum2}]\label{Sylow}Let $\Gamma\cong (\zet_2)^n$ with $n>1$, then $\Gamma$ has $3$-ZSPP. 
\end{thm}

Note that, in general, even the weaker version of Conjecture~\ref{conjectureT} posed by \citet{CicZ} is still open.
\begin{conj}[\citep{CicZ}]\label{conjecture}Let $\Gamma$ be a finite Abelian group with $|I(\Gamma)|>1$. Then $\Gamma$ has $3$-ZSPP. 
\end{conj}

For every Abelian group $\Gamma$ with more than one involution and large enough order, first~\citet{CicTuz} showed that it has $4$-ZSPP, and later~\citet{MP} showed that it indeed has $3$-ZSPP.
This result was improved by~\citet{CicrSuch}, for every $\Gamma$ of order $2^{n}$ for some integer $n>1$ such that $I(\Gamma)\neq 1$, with the following theorem. 
\begin{thm}[\citep{CicrSuch}]\label{SK}Let $\Gamma$ be such that $|I(\Gamma)|>1$ and $|\Gamma|=2^n$ for some integer $n>1$. Then $\Gamma$ has {$3$-ZSPP.}
\end{thm}

Recall that, for every Abelian group $\Gamma$, there exist groups $L$ and $H$ with $|L|=2^\eta$ for a nonnegative integer $\eta$ and $|H|=\rho$ for an odd positive integer $\rho$ such that $\Gamma\cong L\times H$. Moreover, a finite Abelian group $\Gamma$ has no involutions if and only if its order is odd. So, if $\eta=0$, then $\Gamma$ has $2$-ZSPP by Theorem \ref{Zeng}. On the other hand, if $\rho=1$, then $\Gamma$ has $3$-ZSPP by Theorem \ref{SK}.

Let $\Gamma$ be an Abelian group such that $\Gamma\cong L\times H$, where $|L|=2^\eta$ and $|H|=\rho$ for a positive integer $\eta$ and an odd positive integer $\rho$, with $\rho>1$. The main contribution of this paper is to show that $\Gamma$ has $3$-ZSPP if $(\rho \bmod{6}) \in \{1,3\}$. Moreover, for $(\rho \bmod{6}) = 5$, we show that $\Gamma$ has $4$-ZSPP.  Therefore we improve the result by~\citet{CicTuz} and, towards proving Conjecture~\ref{conjecture}, we leave open only the question if, for $(\rho \bmod{6}) = 5$, $\Gamma$ has not only $4$-ZSPP, but also $3$-ZSPP. We also show that Conjecture~\ref{conjectureT} holds only for a group $\Gamma$ with $|I(\Gamma)| =3$ or $\Gamma \cong (\mathbb{Z}_2)^n$ with any $n>1$. As a complement, we present some applications in irregular, magic-type, and antimagic-type graph labeling to illustrate the relations of zero-sum group partitioning to graph theory.

In proofs, we use zero-sum partitions of {$\left((\zet_2)^3\times\zet_3\right)^*$, $\left((\zet_2)^3\times\zet_5\right)^*$, and $\left((\zet_2)^3\times\zet_7\right)^*$} as base cases. They were analyzed by a computer program that we created, and sample zero-sum partitions that certify that the corresponding sequences are realizable are given in the annexes.

\section{Main Result}\label{sec:mr}

Let Bij$(\Gamma)$ denote the set of all bijections from $\Gamma$ to itself.
A \textit{complete mapping} of a group $A$ is defined as $\varphi\in$Bij$(A)$ such that the mapping $\theta\colon g \mapsto  g^{-1}\varphi(g)$ is also bijective (some authors refer to $\theta$, rather than $\varphi$, as the complete mapping)~\citep{Mann}. In the proof of our main result we will use the following property of complete mappings:
\begin{lem}[\citep{CicZ,Zeng}]\label{bijection}
Let $\Gamma$ be a finite Abelian group such that $|I(\Gamma)|\neq 1$. Then there exist $\phi,\varphi\in$Bij$(\Gamma)$ (not necessarily distinct) such that $g+\phi(g)+\varphi(g)=0$  for every $g\in\Gamma$. In particular, we may assume that $\phi(0)=\varphi(0)=0$.
\end{lem}

In the following, we will consider integer partitions $\{m_i\}_{i=1}^t$ of $m-1$ for some positive integer $t$ with $m_i \geq 3$ for every $1 \leq i \leq t$. Given such a sequence, since $m_i\geq 3$ for every $1 \leq i \leq t$, we can modify the sequence $\{m_i\}_{i=1}^t$ by subdividing each term larger than $5$ into a combination of terms $3$, $4$, and $5$. It is easy to check that if the non-zero elements of a group $\Gamma$ can be partitioned into zero-sum subsets of cardinalities corresponding to the elements of the new sequence, then the same holds for the old sequence. So we can focus only on sequences with $m_i \in \{3, 4, 5\}$ for every $1 \leq i \leq t$.

Let $Z$ be a subset of $\Gamma$. Let $\mathcal{K}$ be a zero-sum subset partition of $Z$ with cardinalities in $\{3,4,5\}$ and let $\alpha = |\{S \in\mathcal{K}\colon |S|=3\}|$, $\beta = |\{S \in\mathcal{K}\colon |S|=4\}|$ and $\gamma = |\{S \in\mathcal{K}\colon |S|=5\}|$. In this context, we say that $\mathcal{K}$ \textit{realizes} $(\alpha, \beta, \gamma)$ in $Z$. If there exists a family realizing $(\alpha, \beta, \gamma)$ in $Z$, we say that $(\alpha, \beta, \gamma)$ is \textit{realizable} in $Z$. Note that a triple $(\alpha, \beta, \gamma)$ can be seen as a compact representation of a sequence $\{m_i\}_{i=1}^{t}$ in which $\alpha$ elements are equal to $3$, $\beta$ elements are equal to $4$, and $\gamma$ elements are equal to $5$. Recall that the quotient group $\Gamma$ modulo $H$ for a subgroup $H$ of $\Gamma$ is denoted by $\Gamma/H$.

The following three theorems lead to the main results of the paper, which is given in Corollary~\ref{nowemain}, that states that the $4$-Zero-Sum Partition Property holds for a finite Abelian group $\Gamma$ if and only if $|I(\Gamma)|\neq 1$.

\begin{thm}
Let $\Gamma$ be a finite Abelian group such that $\Gamma\cong L\times H$, with $|L|=2^{\eta}$ for some positive integer $\eta$, $|I(L)|> 1$, and $|H|\equiv 1\pmod 6$. Then $\Gamma$ has {$3$-ZSPP.}\label{1mod6}
\end{thm}
\begin{proof}
By Theorem~\ref{Zeng}, we can assume that $I(L)\geq 7$ and, by Theorem~\ref{SK}, that $|H|>1$.

By Theorem~\ref{Tannenbaum1}, there exists a partition of $H^*$ into good $6$-subsets:
$$H = \{0\}\cup\bigcup\limits_{i=1}^{|H^*|/6}\{b_i, c_i,-b_i - c_i,-b_i,-c_i, b_i + c_i\}.$$

By Theorem~\ref{SK}, $L$ has $3$-ZSPP and  there exist $\phi,\varphi\in$Bij$(L)$ such that $a+\phi(a)+\varphi(a)=0$ for every $a\in L$ by Lemma~\ref{bijection}. Thus
$$\Gamma^* = \bigcup \limits_{a\in L^*}\{(a,0)\}$$ $$\cup \bigcup \limits_{a\in L}\bigcup\limits_{i=1}^{|H^*|/6}\{(a,b_i),(\phi(a), c_i),(\varphi(a),-b_i-c_i),(-a,-b_i),(-\phi(a),-c_i),(-\varphi(a), b_i+c_i)\},$$
where the $6$-subsets are good. Let $M^*=\cup_{a\in L^*}\{(a,0)\}$,  note that $M^* \cong L^*$, thus it has $3$-ZSPP. Let 
$$W=\bigcup \limits_{a\in L}\bigcup\limits_{i=1}^{|H^*|/6}\{(a,b_i),(\phi(a), c_i),(\varphi(a),-b_i-c_i),(-a,-b_i),(-\phi(a),-c_i),(-\varphi(a), b_i+c_i)\}.$$

 We will prove that any triple $(\alpha,\beta,\gamma)$ such that $3\alpha+4\beta+5\gamma=|\Gamma|-1$ is realizable in $\Gamma^*$. 
 Assume that $r_1,\dots,r_{\alpha}=3$, $r_{\alpha+1},\ldots,r_{\alpha+\gamma}=5$, and $r_{\alpha+\gamma+1},\dots,r_{\alpha+\beta+\gamma}=4$. Let $l$ be such that $\sum_{i=1}^{l-1} r_i\leq|L^*|$ and $\sum_{i=1}^{l} r_i> |L^*|$.
Let $r_{l}'=|L^*|- \sum_{i=1}^{l-1} r_i$ and $r_l''=r_l-r_l'$. 
If ($r_l'=0$ or $r_l'\geq 3$) and   $r_l''\geq 2$, then the sequence
 $r_1,\ldots,r_{l-1},r_{l}'$ is realized by a zero-sum partition $A_1,\ldots,A_{l-1},A_{l}'$ of $M^*$ by Theorem~\ref{SK},
 and the sequence $r_{l}'',r_{l+1},r_{l+2},\ldots,r_{\alpha+\beta+\gamma}$ is realized by a zero-sum partition
  $A_{l}'',A_{l+1},A_{l+2},\ldots,A_{\alpha+\beta+\gamma}$ of $W$ since $W$ is the union of good subsets and $|W|=\sum_{i=l+1}^{\alpha+\beta+\gamma}r_i+r_l''$. So we are done. 
Hence we have to settle the cases where $r''_l=1$ or $r'_l=1$ or $r'_l=2$.

\textit{Case 1}  $r_l''=1$.
 
Then $l$ is even and $r_{l+1}$ is odd, since $|L^*|$ is odd and $|W|$ is even (thus $|W|-r''_l$ requires at least one odd term among $r_{l+1},\dots,r_{\alpha+\beta+\gamma}$, and all odd terms are put at the beginning of the sequence).
Suppose first $r_l=3$, then $r_{l}'=2$ and, since $|I(L)|\geq 7$ and $|L|$ is even, there is $|L^*|>7$. This implies that $r_{l-1}=r_{l-2}=3$. Note that, because $|W|\equiv0\pmod 3$, there has to be $\gamma\geq 1$ or $\beta\geq 2$. If  $\gamma\geq 1$, then $(l-2,0,1)$ is realizable in $M^*$ and $(\alpha-l+2,\beta,\gamma-1)$ is realizable in $W$. For $\beta\geq 2$, the triple $(l-3,2,0)$ is realizable in $M^*$ and $(\alpha-l+3,\beta-2,\gamma)$ is realizable in $W$.
Suppose that $r_l=5$. If $\beta>0$, then the triple $(\alpha,1,l-\alpha-1)$ is realizable in $M^*$ and $(0,\beta-1,\gamma+\alpha-l+1)$ is realizable in $W$.

Assume now that $\beta=0$. If $|L^*|>7$, then, for $\alpha\geq 2$, the triple $(\alpha-2,0,l-\alpha+1)$ is realizable in $M^*$ and $(2,0,\gamma+\alpha-l-1)$ is realizable in $W$. For $\alpha\leq1$, there is $r_{l-1}=5$, and $(\alpha+3,0,l-\alpha-2)$ is realizable in $M^*$ and $(1,2,\gamma+\alpha-l-2)$ is realizable in $W$.
Observe that, in this case, by the construction of $W$, the two zero-sum $4$-subsets in $W$ can be split into four zero-sum $2$-subsets and we obtain a realization of $(\alpha,\beta,\gamma)$ in $\Gamma^*$. If $|L^*|=7$, then $\alpha=1$ and $L\cong(\zet_2)^3$, which implies that {$|\Gamma|=8|H|= 4+5\gamma$.} Therefore, $\Gamma \cong (\zet_2)^3\times H$ with $|H|\geq13$, $\gamma=(|\Gamma|-4)/5=(8|H|-4)/5\geq 20$.   
Take the following three zero-sum $5$-subsets:
\begin{center}
$\{(1,0,0,0), \ \ (0,1,0,0) , \ \ (1,1,1,0) , \ \ (0,0,1,b_1), \ \ (0,0,0,-b_1)\}$,

$\{(1,1,0,0) , \ \ (1,0,1,0) , \ \ (0,1,1,0) , \ \ (1,0,0,b_1) , \ \ (1,0,0,-b_1)\}$, 

$\{(0,0,1,0) , \ \ (0,0,1,-b_1), \ \ (0,0,0, b_1), \ \ (0,1,0,b_1) , \ \ (0,1,0,-b_1)\}$.
\end{center}

They cover $M^*$ and four inverse pairs of elements from $W$.
 $W'$ obtained by the deletion of these 15 elements from $W$ is of the following form (let $i_0=(0,0,0)$, $i_1=(1,0,0)$, $i_2=(0,1,0)$ and $i_3=(0,0,1)$):
$$W'=\bigcup \limits_{a\in\{i_0,i_1,i_2,i_3\}}\{(\phi(a),c_1),(\varphi(a),-b_1-c_1),(-\phi(a),-c_1), (-\varphi(a),b_1+c_1)\}$$
$$\cup\bigcup\limits_{a\in (\zet_2)^3\setminus\{i_0,i_1,i_2,i_3\}}\{(a,b_1),(\phi(a), c_1),(\varphi(a),-b_1-c_1),(-a,-b_1),(-\phi(a),-c_1),(-\varphi(a), b_1+c_1)\}$$
$$\cup\bigcup\limits_{a\in (\zet_2)^3}\bigcup\limits_{i=2}^{|H^*|/6}\{(a,b_i),(\phi(a), c_i),(\varphi(a),-b_i-c_i),(-a,-b_i),(-\phi(a),-c_i),(-\varphi(a), b_i+c_i)\}.$$
Note there are $4\frac{|H|-4}{3}>\gamma$ zero-sum $6$-subsets in $W'$, thus we are able to find $(\gamma-2)$ zero-sum 3-subsets in $W'$. The remaining elements of $W'$ form zero-sum $2$-subsets.
Thus we can acquire a realization of $(1,0,\gamma)$ in $\Gamma^*$.

 \textit{Case 2}  $r_l'=2$.
 
If $r_l=3$, then we are done as in Case 1. Therefore, we can assume that $r_l=5$ (for $r_l=4$, we get a contradiction since $|W|$ is divisible by $4$). For $\alpha\geq1$, we obtain that $(\alpha-1,0,l-\alpha)$ is realizable in $M^*$ and $(1,\beta,\gamma+\alpha-l)$ is realizable in $W$. For $\beta\geq 1$, since $|L^*|\geq7$, there exists a realization of $(1,1,l-2)$ in $M^*$ and a realization of $(1,\beta,\gamma-l)$ in $W$. By the construction of $W$, one zero-sum $4$-subset in $W$ can be split into two zero-sum $2$-subsets and we obtain a realization of $(0,\beta,\gamma)$ in $\Gamma^*$.
The only missing case is $\alpha=\beta=0$. If $|L^*|>7$, then actually $|L^*|\geq 15$ and $r_{l-1}=r_{l-2}=5$. Thus $(4,0,l-3)$ is realizable in $M^*$. Since $W$ is a {union} of good $6$-subsets, there exists a partition into four zero-sum $2$-subsets and $(\gamma-l-1)$ zero-sum $5$-subsets.

Assume now that $|L^*|=7= |I(L)|$. But then, as in Case 1, there is $\Gamma \cong (\zet_2)^3\times H$ and we proceed the same way.

 \textit{Case 3.}  $r_l'=1$.
 
Then $l> 1$, and $r_l''\geq 2$ is even. 
Assume first that $r_l=3$. Since $|W|\equiv 0\pmod 3$, there is $\beta+\gamma>0$.
If now $\beta>0$, then there exists a realization of $(l-2,1,0)$ in $M^*$ and a realization of $(\alpha-l+2,\beta-1,\gamma)$ in $W$. Therefore, we can assume that $\beta=0$. If now $\gamma=1$, then  $|W|\equiv 1\pmod 3$, a contradiction. Thus, for $\beta=0$, there is $\gamma\geq 2$. If $|L^*|>7$, then $|L^*|\geq 15$ and $r_{l-1}=r_{l-2}=r_{l-3}=3$ and   $(l-4,0,2)$ is realizable in $M^*$, and $(\alpha-l+4,0,\gamma-2)$ is realizable in $W$. For $|L^*|=7= |I(L)|$, there is $\Gamma \cong (\zet_2)^3\times H$.
The set 
$$\smashoperator[r]{\bigcup\limits_{a\in ((\zet_2)^3)^*}}\{(a,0)\} \cup \smashoperator{\bigcup\limits_{a\in (\zet_2)^3}}\{(a,b_1),(\phi(a), c_1),(\varphi(a),-b_1-c_1),(-a,-b_1),(-\phi(a),-c_1),(-\varphi(a), b_1+c_1)\}$$
can be partitioned into $15$ zero-sum $3$-subsets and $2$ zero-sum $5$-subsets in the following way:
$$\{(0, 0, 1, 0),\\ (0, 0, 0, b_1),\\ (0, 0, 1, -b_1)\},\;\;\;\{(1, 0, 0, 0),\\ (0, 0, 1, b_1), \\(1, 0, 1, -b_1)\},$$
$$\{(1, 1, 1, 0),\\ (0, 1, 1, b_1),\\ (1, 0, 0, -b_1)\},\;\;\;\{(0, 1, 0, b_1),\\ (0, 0, 0, c_1), \\(0, 1, 0, -b_1-c_1)\},$$
$$\{(1, 0, 1, b_1),\\ (0, 0, 1, c_1),\\ (1, 0, 0, -b_1-c_1)\},\;\;\;\{(0, 1, 0, 0),\\ (0, 1, 0, c_1),\\ (0, 0, 0, -c_1)\},$$
$$\{(0, 0, 0, b_1+c_1),\\ (0, 1, 0, -c_1), \\(0, 1, 0, -b_1)\},\;\;\;\{(0, 0, 1, b_1+c_1),\\ (0, 0, 1, -c_1),\\ (0, 0, 0, -b_1)\},$$
$$\{(0, 1, 0, b_1+c_1),\\ (1, 0, 0, -c_1),\\ (1, 1, 0, -b_1)\},\;\;\;\{(1, 0, 0, b_1+c_1),\\ (0, 1, 1, -c_1),\\ (1, 1, 1, -b_1)\},$$
$$\{(1, 0, 0, b_1), (1, 1, 1, c_1), (0, 1, 1, -b_1-c_1)\},\;\;\;\{(1, 1, 1, b_1), (1, 1, 0, c_1), (0, 0, 1, -b_1-c_1)\},$$
$$\{(1, 1, 0, b_1), (0, 1, 1, c_1), (1, 0, 1, -b_1-c_1)\},\;\;\;\{(1, 0, 1, b_1+c_1),\\ (1, 1, 0, -c_1),\\ (0, 1, 1, -b_1)\},$$
$$\{(0, 1, 1, 0),\\ (1, 0, 0, c_1),\\ (1, 1, 1, -c_1)\},$$
$$\{(1, 0, 1, 0),\\ (1, 0, 1, c_1),\\ (0, 1, 1, b_1+c_1),\\ (1, 1, 0, -b_1-c_1), (1, 0, 1, -c_1)\},$$
$$\{(1, 1, 0, 0),\\ (1, 1, 0, b_1+c_1),\\ (1, 1, 1, b_1+c_1),\\ (0, 0, 0, -b_1-c_1), (1, 1, 1, -b_1-c_1)\}.$$
The cases $(10,0,5)$ and $(5,0,8)$ for $(\zet_2)^3\times\zet_7$ were analyzed by a computer program we created, sample realizations can be found in the annexes. Thus we can assume that $|H|\geq 13$ and the nonempty
set $\cup_{a\in (\zet_2)^3}\cup_{i=2}^{|H^*|/6}\{(a,b_i),(\phi(a), c_i),(\varphi(a),-b_i-c_i),(-a,-b_i),(-\phi(a),-c_i),(-\varphi(a), b_i+c_i)\}$ is a union of good $6$-subsets. Hence we obtain that $(\alpha,0,\gamma)$ is realizable in $\Gamma^*$.

Suppose now that $r_l=5$. If $r_{l-1}=5$, then let $r_{l-1}'=r_{l-1}-2\geq3$ and
 $r_{l-1}''=2$. Let us re-define $r_l':=3$ and $r_{l}'':=r_l-3$.
The sequence $r_1,\ldots,r_{l-2},r_{l-1}',r_l'$ is realized by a zero-sum partition $A_1,\ldots,A_{l-2},A_{l-1}',A_l'$ of $M^*$ by Theorem~\ref{SK}, and the sequence
  $r_{l-1}'',r_l'',r_{l+1},\ldots,r_{\alpha+\beta+\gamma}$ is realized by a zero-sum partition $A_{l-1}'',A_l'', A_{l+1},\ldots,A_{\alpha+\beta+\gamma}$ of $W$. Thus $r_{l-1}=3$. If now $\beta>0$, then there exists a realization of $(\alpha-1,1,0)$ in $M^*$ and a realization of $(1,\beta-1,\gamma)$ in $W$. If $|L^*|>7$, then $|L^*|\geq 15$ and $r_{l-1}=r_{l-2}=r_{l-3}=3$. Since $\beta=0$ and $\gamma\geq 2$, we obtain that $(\alpha-3,0,2)$ is realizable in $M^*$, and $(3,0,\gamma-2)$ is realizable in $W$. Let $|L^*|=7= |I(L)|$. Then $L\cong (\zet_2)^3$, $\alpha=2$ and we proceed as above to obtain a realization of the sequence $(2,0,\gamma)$.
\end{proof}

We will need now the following lemma~\citep[Lemma 2.4]{CicrSuch}:
\begin{lem}[\citep{CicrSuch}]\label{l3.12} {Let $\Gamma\cong A\times \zet_{2^{n_1}}$ be an Abelian group such that $|A|=4n$, $|I(A)|>1$, $n_1\in\{1,2\}$, and $n$ is a positive integer. Let $\alpha$, $\beta$, $\gamma$ be non-negative integers with $3\alpha+4\beta+5\gamma=|\Gamma|-1$ and $\beta\geq (2^{n_1}-1)n$. If there exists a subset partition $\mathcal{K}$ realizing $(\alpha, \beta-(2^{n_1}-1)n, \gamma)$ in $A^{*}$, then $(\alpha, \beta, \gamma)$ is realizable in $\Gamma^{*}$. } \end{lem}

\begin{thm}
Let $\Gamma$ be a finite Abelian group such that $\Gamma=L\times H$ such that  $|L|=2^{\eta}$ for some natural number $\eta$, $|I(L)|> 1$ and $|H|\equiv 3\pmod 6$. Then $\Gamma$ has {the $3$-Zero-Sum Partition Property.}\label{3mod6}
\end{thm}
\begin{proof}
By Theorem~\ref{Zeng}, we can assume that $|I(L)|\geq 7$. By Theorem~\ref{SK}, $L$ has $3$-ZSPP and  there are
 $\phi,\varphi\in$Bij$(L)$ such that $a+\phi(a)+\varphi(a)=0$ for every $a\in L$ and $\phi(0)=\varphi(0)=0$ by Lemma~\ref{bijection}.

\textit{Case 1.} $|H|>3$.

By Theorem~\ref{Tannenbaum1} there exists a Skolem partition of $H^*$:
$$H = \{0,b,-b\}\cup\bigcup\limits_{i=1}^{(|H^*|-2)/6}\{b_i, c_i,-b_i - c_i,-b_i,-c_i, b_i + c_i\}.$$

Let $$W= \bigcup \limits_{a\in L}\bigcup\limits_{i=1}^{(|H^*|-2)/6}\{(a,b_i),(\phi(a), c_i),(\varphi(a),-b_i-c_i),(-a,-b_i),(-\phi(a),-c_i),(-\varphi(a), b_i+c_i)\}$$ and $M^*=\cup_{a\in L^*} \{(a,0)\}$.

 Thus

$$\Gamma^* =\{(0,b),(0,-b)\} \cup \bigcup \limits_{a\in L^*}\{(a,0),(\phi(a),b),(\varphi(a),-b)\}\cup W.$$

We will prove that any triple $(\alpha,\beta,\gamma)$ such that $3\alpha+4\beta+5\gamma=|\Gamma|-1$ is realizable in $\Gamma^*$. If $\alpha+\gamma\geq |L^*|$, since  $(|L^*|,0,0)$ is realizable in $\cup_{a\in L^*}\{(a,0),(\phi(a),b),(\varphi(a),-b)\}$ and $W$ is a union of good $6$-subsets, we have that $(\alpha,\beta,\gamma)$ is realizable in $\Gamma^*$. Thus we assume that $\alpha+\gamma< |L^*|$. Observe that

$$\Gamma^* =M^*\cup \bigcup \limits_{a\in L}\{(a,b),(-a,-b)\}\cup W.$$

Since $|L^*|\equiv 3\pmod 4$, $M^*\cong L^*$,  the sequence $(1,(|L^*|-3)/4,0)$ is realizable in $M^*$. Note that $|W| = |L|(|H|-3) \geq 6|L| > 3(\alpha+\gamma)$. This implies the existence of a partition of $W$ into $(\alpha+\gamma-1)$ zero-sum 3-sets and $(|L|(|H|-3)-3(\alpha+\gamma-1))/2$ zero-sum 2-sets. Furthermore, within the set $\cup_{a\in L}\{(a,b),(-a,-b)\}$, there are $|L| > \gamma$ zero-sum 2-subsets. Hence, a realization of $(\alpha, \beta, \gamma)$ exists in $\Gamma^*$.

\textit{Case 2.} $|H|=3$.

Thus $H\cong \zet_3$ and
$$\Gamma^* =M^*\cup \bigcup \limits_{a\in L}\{(a,1),(-a,2)\}.$$

Assume first that $\Gamma^*\cong\left((\zet_2)^3\times\zet_3\right)^*$. By Theorem~\ref{Zeng}, $(\zet_2)^2\times\zet_3$ has $2$-ZSPP. The triples $(\alpha, \beta, \gamma)$ with $\beta<3$ were analyzed by a computer program we created, sample realizations can be found in the annexes. Therefore, for $A=(\zet_2)^2\times\zet_3$ and  $n_1=1$, by Lemma~\ref{l3.12}, we obtain that every triple $(\alpha, \beta, \gamma)$ is realizable in $\left((\zet_2)^3\times\zet_3\right)^*$.

We will prove now that any triple $(\alpha,\beta,\gamma)$ such that $3\alpha+4\beta+5\gamma=|\Gamma|-1$ is realizable in $\Gamma^*\not \cong \left((\zet_2)^3\times\zet_3\right)^*$.
If $3(\alpha+\gamma)\leq|L^*|$, then, for $|L^*|-3(\alpha+\gamma)\equiv0\pmod4$, the sequence $(\alpha+\gamma,(|L^*|-3(\alpha+\gamma))/4,0)$ is realizable in $M^*$ by Theorem~\ref{SK}. We will obtain $\gamma$ zero-sum $5$-subsets by taking $\gamma$ (possibly $\gamma=0$) zero-sum $2$-subsets from the set $\cup_{a\in L}\{(a,1),(-a,2)\}$. Assume now that $|L^*|-3(\alpha+\gamma)\equiv 2\pmod4$. But there is $|L^*| - 3(\alpha + \gamma) = |\Gamma^*| - 2|L| - 3(\alpha+\gamma) = 4\beta - 2|L| \equiv 0 \pmod 4$, which yields a contradiction. The only missing case is $3(\alpha+\gamma)>|L^*|$. Since $L\not \cong (\zet_{2})^3$, there exists a subgroup $B$ of $\Gamma$ such that $|B|=2^{\eta-2}=|L|/4$,, $|I(B)|>1$ and $\Gamma/B \cong \zet_2\times \zet_6$. Since $\zet_2\times \zet_6=\{(0,0),(0,2),(0,4)\}\cup\{(0,3),(1,3),(1,0)\}\cup\{\{(0,1),(1,4),(1,1),(0,5),(1,2),(1,5)\}$, we can choose a set of coset representatives for the subgroup $B$ in $\Gamma$, say $A$, such
that $$A = \{0\}\cup\{a,-a\} \cup \{ e_1,e_2,-e_1-e_2\}\cup\{c,d, -c-d,-c,-d,c+d\},$$
where $2e_1,2e_2 \in B$. Since  $|I(B)|>1$, there exist {$\phi,\varphi\in$Bij$(B)$ such that $g+\phi(g)+\varphi(g)=0$  for every $g\in B$} by Lemma~\ref{bijection}. Moreover, without loss of generality, we can assume that  $\phi(0)=\varphi(0)=0$. Note that $|B|=|\Gamma|/12=|L|/4$. If now $\alpha+\gamma< |L|/2-1$, then write  $$\begin{array}{ll}
\Gamma^* =&  B^* \cup\bigcup\limits_{\substack{g\in B}}\{ a+g,-a-g)\} \cup\bigcup\limits_{\substack{g\in B}}\{ e_1+g,e_2+\phi(g),-e_1-e_2+\varphi(g)\}\\
& \cup \bigcup\limits_{g\in B}\left\{c+g, d+\phi(g),-c-d+\varphi(g),-c-g,-d-\phi(g), c+d-\varphi(g)\right\},\end{array}$$
where the latter $6$-subsets are good.

If $|B^*|$ is divisible by $3$, then the sequence $(|B^*|/3,0,0)$ is realizable in  $B^*$ by Theorem~\ref{SK} and $(|B|,0,0)$ is realizable in  $\cup_{\substack{g\in B}}\{ e_1+g,e_2+\phi(g),-e_1-e_2+\varphi(g)\}$. Since $|B^*|/3+|B|=(|L|-1)/3<\alpha +\gamma$, we can find the remaining $\alpha +\gamma- (|L|-1)/3\leq 2|B|$ zero-sum $3$-subsets in 
$$\bigcup\limits_{g\in B}\left\{c+g, d+\phi(g),-c-d+\varphi(g),-c-g,-d-\phi(g), c+d-\varphi(g)\right\}.$$ Note that we are left now only with zero-sum $2$-subsets, thus we are done.

If $|B^*|$ is not divisible by $3$, then, since $|B^*|\equiv 3\pmod 4$ and $|B^*| = 2^{\eta-2} - 1 \not\equiv 2 \pmod 3$, there is $|B^*|\equiv 7\pmod {12}$. Moreover $\beta>1$, because $|\Gamma|\not \equiv 0 \pmod 5$ and  $\alpha+\gamma< |L|/2-1<(3|L|-1)/5=|\Gamma^*|/5$. Thus the sequence $((|B^*|-4)/3,1,0)$ is realizable in $B^*$ by Theorem~\ref{SK} and we proceed as above.
 
For $\alpha+\gamma\geq |L|/2-1$, let
$$\begin{array}{ll}
\Gamma^* =
& \{ a,-a\}\cup\bigcup\limits_{\substack{g\in B^*}}\{ a+g,\phi(g),-a+\varphi(g)\} \cup\bigcup\limits_{\substack{g\in B}}(\{ e_1+g,e_2+\phi(g),-e_1-e_2+\varphi(g)\}\\
& \cup
\bigcup\limits_{g\in B}\left\{c+g, d+\phi(g),-c-d+\varphi(g),-c-g,-d-\phi(g), c+d-\varphi(g)\right\}.\end{array}$$

The sequence $(|B|-1,0,0)$ is realizable in  $\cup_{g\in B^*}\{ a+g,\phi(g),-a+\varphi(g)\}$ and $(|B|,0,0)$ is realizable in $\cup_{g\in B}\{ e_1+g,e_2+\phi(g),-e_1-e_2+\varphi(g)\}$. Since $ |B|-1+|B|=|L|/2-1\leq\alpha +\gamma$, we can find the remaining $\alpha +\gamma- |L|/2+1$ zero-sum $3$-subsets in 
$$\bigcup\limits_{g\in B}\left\{c+g, d+\phi(g),-c-d+\varphi(g),-c-g,-d-\phi(g), c+d-\varphi(g)\right\}.$$ Note that we are left now only with zero-sum $2$-subsets, thus we are done.
\end{proof}

\begin{thm}
Let $\Gamma$ be a finite Abelian group such that $\Gamma=L\times H$ such that  $|L|=2^{\eta}$ for some natural number $\eta$, $|I(L)|> 1$, and $|H|\equiv 5\pmod 6$. Then $\Gamma$ has {the $4$-Zero-Sum Partition Property.}\label{5mod6}
\end{thm}
\begin{proof}
By Theorem~\ref{Zeng}, we can assume that $|I(L)|\geq 7$. 
By Theorem~\ref{SK}, the group $L$ has $3$-ZSPP and there exist $\phi,\varphi\in$Bij$(L)$ such that $a+\phi(a)+\varphi(a)=0$ for every $a\in L$, and $\phi(0)=\varphi(0)=0$ by Lemma~\ref{bijection}. Let $r_1+\ldots +r_t=|\Gamma|-1$.
Assume that $r_1,\dots,r_s$ are all odd and $r_{s+1},\dots,r_t$ are all even.
Then $s$ is odd because $|\Gamma|$ is even. We may restrict ourselves to sequences where $4 \leq r_i \leq 7$ for every $i$. We refine the sequence $r_1,\dots,r_s$ to $r_1'=\ldots=r_s'=3$. Note that $s< |\Gamma|/5$.
 
 \textit{Case 1.} $|H|>5$.

By Theorem~\ref{Tannenbaum1}, there exists a Skolem partition of $H^*$:
$$H = \{0,b,-b,c,-c\}\cup\bigcup\limits_{i=1}^{(|H^*|-4)/6}\{b_i, c_i,-b_i - c_i,-b_i,-c_i, b_i + c_i\}.$$

Let $$W= \bigcup \limits_{a\in L}\bigcup\limits_{i=1}^{(|H^*|-4)/6}\{(a,b_i),(\phi(a), c_i),(\varphi(a),-b_i-c_i),(-a,-b_i),(-\phi(a),-c_i),(-\varphi(a), b_i+c_i)\}$$ and $M^*=\cup_{a\in L^*} \{(a,0)\}$.
Thus

$$\Gamma^* =\{(0,b),(0,-b)\} \cup \bigcup \limits_{a\in L^*}\{(a,0),(\phi(a),b),(\varphi(a),-b)\}\cup\bigcup \limits_{a\in L}\{(a,c),(-a,-c)\}\cup W.$$

If $s \geq |L^*|$, then we can realize $|L^*|$ zero-sum $3$-subsets in $\cup_{a\in L^*}\{(a,0),(\phi(a),b),(\varphi(a),-b)\}$. Since $s<|\Gamma|/5$, the remaining $s-|L^*|$ zero-sum $3$-subsets can be realized in $W$, and the remaining elements in $\Gamma^*$ form zero-sum $2$-subsets, so we are done.



Assume now that $s<|L^*|$. Note that we can also write $\Gamma^*$ in the following way:

$$\Gamma^* =M^*\cup \bigcup \limits_{a\in L}\{(a,b),(-a,-b)\}\cup \bigcup \limits_{a\in L}\{(a,c),(-a,-c)\}\cup W.$$

Since $|L^*|\equiv 3\pmod 4$ and $L^*\cong \cup_{a\in L^*} \{(a,0)\}$,  the sequence $(1,(|L^*|-3)/4,0)$ is realizable in $M^*$ by Theorem~\ref{SK}. Observe that $|W|=|L|(|H|-5)\geq 6|L|>2(3s)$. This implies that there exists a partition of $W$ into $s-1$ zero-sum 3-sets and $(|L|(|H|-5)-3(s-1))/2$ zero-sum 2-sets. Moreover, the set $\left(\cup_{a\in L}\{(a,b),(-a,-b)\}\right)\cup \left(\cup_{a\in L}\{(a,c),(-a,-c)\}\right)$ contains $2|L|>2s$ distinct zero-sum $2$-sets. Thus there exists the desired subset partition of $\Gamma^*$.

\textit{Case 2.} $|H|=5$.

Thus $H\cong \zet_5$ and for non-zero $b,c\in\zet_5$ such that $b\neq c$, $b\neq -c$ there is
$$\Gamma^* =M^*\cup \bigcup\limits_{a\in L}\{(a,b),(-a,-b)\}\cup \bigcup\limits_{a\in L}\{(a,c),(-a,-c)\}.$$
Assume first that $\Gamma\cong\left((\zet_2)^3\times\zet_5\right)^*$. By Theorem~\ref{Zeng}, $\Gamma\cong(\zet_2)^2\times\zet_5$ has $2$-ZSPP. The triples $(\alpha, \beta, \gamma)$ with $\beta<5$ were analyzed by a computer program we created, sample realizations can be found in the annexes. Therefore, for $A=(\zet_2)^2\times\zet_5$ and  $n_1=1$, by Lemma~\ref{l3.12}, we obtain that every triple $(\alpha, \beta, \gamma)$ is realizable in $\Gamma^*\cong\left((\zet_2)^3\times\zet_5\right)^*$.

From now on, we assume that $L\not\cong (\zet_2)^3$. If $3s\leq|L^*|$, then, for $|L^*|-3s\equiv0\pmod4$, the sequence $(s,(|L^*|-3s)/4,0)$ is realizable in $M^*\cong L^*$ by Theorem~\ref{SK}. We will obtain $s$ zero-sum subsets of odd cardinality by taking zero-sum $2$-subsets from the set $\left(\cup_{a\in L}\{(a,b),(-a,-b)\}\right)\cup \left(\cup_{a\in L}\{(a,c),(-a,-c)\}\right)$. If $|L^*|-3s\equiv 2\pmod4$, then the sequence $(s-1,(|L^*|-3s-2)/4,1)$ is realizable in $ M^*$ by Theorem~\ref{SK}, and we proceed as before. 

Now consider $3s>|L^*|$. Since $L\not \cong (\zet_{2})^3$, there exists a subgroup $B$ of $\Gamma$ such that $|B|=2^{\eta-2}=|L|/4$, $|I(B)|>3$ and $\Gamma/B \cong \zet_2\times \zet_{10}$. Since $\zet_2\times \zet_{10}=\{(0,0),(0,5),(1,5),(1,0)\}\cup\{(0,2),(0,8)\}\cup\{(0,4),(0,6)\}\cup\{(0,1),(1,3),(1,6),(0,9),(1,7),(1,4)\}\cup\{(1,1),(1,2),(0,7),(1,9),(1,8),(0,3)\}$ we can choose a set of coset representatives for the subgroup $B$ in $\Gamma$, say $A$, such
that 
$$A = \{a,0,-a\}\cup\{b,-b\} \cup \{ e_1,e_2,-e_1-e_2\}\cup\bigcup_{i=1,2}\{c_i,d_i, -c_i-d_i,-c_i,-d_i,c_i+d_i\},$$
 where $2e_1,2e_2 \in B$. Since  $|I(B)|>1$, there are
 {$\phi,\varphi\in$Bij$(B)$ such that $g+\phi(g)+\varphi(g)=0$ for every $g\in B$} by Lemma~\ref{bijection}. Moreover, without loss of generality, we can assume that  $\phi(0)=\varphi(0)=0$.
 If now $s< |L|/2-1$, then write  $$\begin{array}{ll}
\Gamma^* =& B^* \cup\bigcup\limits_{\substack{g\in B}}\{ a+g,-a-g\}\cup\bigcup\limits_{\substack{g\in B}}\{ b+g,-b-g\}\\& \cup\bigcup\limits_{\substack{g\in B}}\{ e_1+g,e_2+\phi(g),-e_1-e_2+\varphi(g)\}\\ 
&\cup\bigcup\limits_{\substack{g\in B\\i=1,2}}\left\{c_i+g, d_i+\phi(g),-c_i-d_i+\varphi(g),-c_i-g,-d_i-\phi(g), c_i+d_i-\varphi(g)\right\},\end{array}$$
where the latter $6$-subsets are good.

If now $|B^*|$ is divisible by $3$, then the sequence $(|B^*|/3,0,0)$ is realizable in  $ B^*$ by Theorem~\ref{SK} and $(|B|,0,0)$ is realizable in  $\cup_{g\in B}\{ e_1+g,e_2+\phi(g),-e_1-e_2+\varphi(g)\}$. Since $|B^*|/3+|B|=(|L|-1)/3<s$, we can find the remaining $s- (|L|-1)/3\leq 2|B|$ zero-sum $3$-subsets in 
$$\bigcup\limits_{g\in B}\left\{c_1+g, d_1+\phi(g),-c_1-d_1+\varphi(g),-c_1-g,-d_1-\phi(g), c_1+d_1-\varphi(g)\right\}.$$ Note that we are left now only with zero-sum $2$-subsets, thus we are done.

If $|B^*|$ is not divisible by $3$, then $|B^*|\equiv 7\pmod {12}$ and  the sequence $((|B^*|-4)/3,1,0)$ is realizable in  $ B^*$ by Theorem~\ref{SK} and we proceed as above.
 
For $s\geq |L|/2-1$, let
$$\begin{array}{ll}
\Gamma^* =& \{ a,-a\}\cup\bigcup\limits_{\substack{g\in B^*}}\{ a+g,\phi(g),-a+\varphi(g)\}\cup\bigcup\limits_{\substack{g\in B}}\{ b+g,-b-g\}\\& \cup\bigcup\limits_{\substack{g\in B}}\{ e_1+g,e_2+\phi(g),-e_1-e_2+\varphi(g)\}\\ 
&\cup\bigcup\limits_{\substack{g\in B\\i=1,2}}\left\{c_i+g, d_i+\phi(g),-c_i-d_i+\varphi(g),-c_i-g,-d_i-\phi(g), c_i+d_i-\varphi(g)\right\},\end{array}$$
where the latter $6$-subsets are good.

Since $|B|=|\Gamma|/20=|L|/4$, $(|L|/2-1,0,0)$ is realizable in $\left(\cup_{g\in B^*}\{a+g,\phi(g),-a+\varphi(g)\}\right)$ $\cup\left(\cup_{g\in B}\{ e_1+g,e_2+\phi(g),-e_1-e_2+\varphi(g)\}\right)$. Since $s<|L|$ and $0 \leq s-|L|/2+1$, we can find the remaining $s- |L|/2+1\leq|L|/2=2|B|$ zero-sum $3$-subsets in $$\bigcup\limits_{\substack{g\in B\\i=1,2}}\left\{c_i+g, d_i+\phi(g),-c_i-d_i+\varphi(g),-c_i-g,-d_i-\phi(g), c_i+d_i-\varphi(g)\right\}.$$ Note that we are left now only with zero-sum $2$-subsets, thus we are done.
\end{proof}

Putting the main results together, we obtain the following corollary. Recall that $|I(\Gamma)| = 1$ if and only if the factorization of $\Gamma$ contains exactly one finite cyclic subgroup of even order.

\begin{cor} A finite Abelian group $\Gamma$  has the $4$-Zero-Sum Partition Property if and only if  $|I(\Gamma)|\neq 1$.\label{nowemain}
\end{cor}
\begin{proof} If $|I(\Gamma)|= 1$, then $\Gamma$ does not have $x$-ZSPP for any $x$ because $\sum_{g\in\Gamma}g\neq 0$.

Assume now that $|I(\Gamma)|\neq 1$. By the fundamental theorem of finite Abelian groups, the group $\Gamma \cong L\times H$ such that  $|L|=2^{\eta}$ for some natural number $\eta$, $|I(L)|\neq 1$, and $|H|$ is odd. If $\eta=0$ then we are done by Theorem~\ref{Zeng} since $|I(\Gamma)|=0$. Assume now that $\eta>0$. So $|I(L)|>1$, and we can apply Theorems~\ref{1mod6}, \ref{3mod6} or \ref{5mod6}.
\end{proof}

\begin{cor}\label{cor_m}
Let $m$ be an integer with $m \geq 4$. Then one of the following cases holds:
\begin{enumerate}
\item $m \bmod{4} = 0$ and there exists an Abelian group $\Gamma$ of order $m$ that has $2$-ZSPP.
\item $m \bmod{4} \in \{1,3\}$ and every  Abelian group $\Gamma$ of order $m$ has $2$-ZSPP.
\item $m \bmod{4} = 2$ and every Abelian group $\Gamma$ of order $m$ does not have $x$-ZSPP for any $x$.
\end{enumerate}
\end{cor}
\begin{proof}
Note that, given a positive integer $m$, if $m \equiv 2 \pmod{4}$, then every group $\Gamma$ of order $m$ has exactly one involution and so does not have $x$-ZSPP for any $x$. But, if $m \geq 4$ and $m \not\equiv 2 \pmod{4}$, then there exists an Abelian group $\Gamma$ of order $m$ that has $2$-ZSPP. Indeed, if $m \bmod{4} \in \{1,3\}$, then $m$ is odd. In this case, $\zet_m$ has no involutions, thus has $2$-ZSPP by Theorem \ref{Zeng}. If $m \bmod{4} = 0$, then $\Gamma = \zet_2 \times \zet_{m/2}$ has exactly $3$ involutions, thus has $2$-ZSPP by Theorem \ref{Zeng}.
\end{proof}

Corollary \ref{nowemain} points in the direction of proving Conjecture \ref{conjecture}. On the other hand, let us show that Conjecture \ref{conjectureT} is false in general, and it is only true in the cases covered by Theorems \ref{Zeng} and \ref{Sylow}.

\begin{thm}\label{Skol}
Let $\Gamma$ be a finite Abelian group with $|I(\Gamma)|>1$. Let $R =\Gamma^* \setminus I(\Gamma)$. For any positive integer $t$ and an integer partition $\{m_i\}_{i=1}^t$ of $|\Gamma^*|$, with {$m_i \geq 2$ for all $i$, $1 \leq i \leq |R|/2$, and {$m_i \geq 3$} for all $i$, $|R|/2+1\leq i\leq t$}, there is a subset partition $\{S_i\}_{i=1}^t$ of $\Gamma^*$ such that $|S_i| = m_i$ and $\sum_{s\in S_i}s = 0$ for all  $1 \leq i \leq t$ if and only if $|I(\Gamma)|\in\{3,|\Gamma^*|\}$.
\end{thm}
\begin{proof}
By  Theorems \ref{Zeng} and \ref{Sylow} we can assume that $|I(\Gamma)| = 2^e - 1 < |\Gamma| - 1$ with $e \geq 3$. 
 
 Let first $e$ be even.
 Let $m_1,\dots,m_{|R|/2-1}=2$, $m_{|R|/2},\ldots,m_{|R|/2+|I(\Gamma)|/3-3}=3$, $m_{|R|/2+|I(\Gamma)|/3-2},$ $m_{|R|/2+|I(\Gamma)|/3-1}=4$.
 Note that $\sum_{i=1}^{|R|/2+|I(\Gamma)|/3-1}m_i=|\Gamma^*|$. Suppose that there is a subset partition $\{S_i\}_{i=1}^t$ of $\Gamma^*$ such that $|S_i| = m_i$ and $\sum_{s\in S_i}s = 0$ for all $i$. Then $R\setminus \left(\bigcup_{i=1}^{|R|/2-1}S_i\right)=\{a,-a\}$ for some $a\in R$. Thus $a\in S_j$ for some $j\in\{|R|/2,\ldots,|R|/2+|I(\Gamma)|/3-1\}$ but then $\sum_{s\in S_j}s\neq 0$ since  $I(\Gamma)\cup\{0\}$ is a subgroup of $\Gamma$. Thus for any $\iota_1,\iota_2,\iota_3\in I(\Gamma)$ there is $\iota_1+\iota_2 \in I(\Gamma)$, $\iota_1+\iota_2+\iota_3 \in I(\Gamma)\cup\{0\}$ but $a\notin I(\Gamma)\cup\{0\}$, a contradiction.

For $e$ odd we can show analogously that there is no partition of $\Gamma^*$ for $m_1,\dots,m_{|R|/2-1}=2$, $m_{|R|/2},$ $\ldots,$ $m_{|R|/2+(|I(\Gamma)|+2)/3-1}=3$.\end{proof}

\section{Some Applications}\label{sec:sa}
In this section, we present some applications of Corollary~\ref{nowemain} to magic- and antimagic-type labelings. Generally speaking, such a labeling of a graph $G=(V,E)$ is a mapping from only $V$ or $E$, or their union $V\cup E$, to a set of labels, which most often is a set of integers or elements of a group. Then the weight of a graph element is typically the sum of labels of the adjacent or incident elements of one or both types. When the weight of all elements is required to be equal, then we speak of a magic-type labeling; when the weights should be all different, then we speak of an antimagic-type labeling.

\subsection{\texorpdfstring{$\Gamma$}{Gamma}-irregular labeling for digraphs}\label{subsec:gild}
Let $\overrightarrow{G} = (V,A)$ be a digraph on the set of vertices $V$ with the set of arcs $A$. If $\psi\colon A \rightarrow \Gamma$ is a labeling of arcs of $\overrightarrow{G}$ with elements of $\Gamma$ such that $\varphi_{\psi}\colon V \rightarrow \Gamma$ and the labeling of vertices of $\overrightarrow{G}$ with elements of $\Gamma$ defined by
$$\varphi_{\psi}(x)=\sum_{y\in N^-(x)}\psi((y,x))-\sum_{y\in N^+(x)}\psi((x,y)),$$ is injective, then we say that $\psi$ is a \textit{$\Gamma$-irregular labeling of $\overrightarrow{G}$}.

{Recall that, given a digraph $\overrightarrow{G}=(V,A)$, its \textit{underlying graph} is the graph $G=(V,E)$ in which $\{x,y\} \in E$ if and only if $(x,y) \in A$ or $(y,x)\in A$. Given a digraph $\overrightarrow{G}$ and its underlying graph $G$, the subdigraph $\overrightarrow{C}$ of $\overrightarrow{G}$ induced by the vertex set of a connected component $C$ in $G$ is called a \textit{weakly connected component} of $\overrightarrow{G}$. Thus we have $\overrightarrow{C} = \overrightarrow{G}[V(C)]$.} 

We have the following result:
\begin{lem}[\citep{CicTuz}]\label{zerosum} A digraph $\overrightarrow{G}=(V,A)$ with no isolated
 vertices has a $\Gamma$-irregular labeling if and only if there exists an injection $\varphi$ from $V$ to $\Gamma$ such that\/
 $\sum_{x\in V(\overrightarrow{C})}\varphi(x)=0$ for every  weakly connected component $\overrightarrow{C}$ of\/ $\overrightarrow{G}$.
\end{lem}

Cichacz and Tuza also showed the following:
\begin{thm}[\citep{CicTuz}]\label{glowne}Any digraph\/ $\overrightarrow{G}$ of order\/ $n$ with no weakly connected components of order less than\/ $3$ has a $\Gamma$-irregular labeling for every\/ $\Gamma$ such that\/
  $|\Gamma|\geq 2n+2\sqrt{n-1/2}-1$.
\end{thm}

They also showed the following:
\begin{lem}[\citep{CicTuz}]\label{lemat} Let $\overrightarrow{G}$ be a digraph of order $n$ with no weakly connected component of the order less than 3,
and let $\Gamma$ be a finite Abelian group such that $|\Gamma|-|I(\Gamma)|\geq2n.$
Then  $\overrightarrow{G}$ has a $\Gamma$-irregular labeling.
\end{lem}

By Corollary~\ref{nowemain}, we obtain the following:
\begin{cor} Any digraph\/ $\overrightarrow{G}$ of order\/ $n$ with no weakly connected components
 of order less than\/ $4$ has a $\Gamma$-irregular labeling for every\/ $\Gamma$ such that\/
  $|I(\Gamma)|\neq 1$ and $|\Gamma|\geq n+5$.\label{liniowe}\end{cor}
\begin{proof}
Let $\{\overrightarrow{C}_i\}_{i=1}^t$ be the weakly connected components of $\overrightarrow{G}$. By Lemma~\ref{zerosum}, there exists a $\Gamma$-irregular labeling of a digraph $\overrightarrow{G}$ with weakly connected components $\{\overrightarrow{C}_i\}_{i=1}^t$ if and only if there exist in $\Gamma$ pairwise disjoint subsets $\{S_i\}_{i=1}^t$ such that $|S_i|=|V(\overrightarrow{C}_i)|$ and $\sum_{s\in S_i}s=0$ for every $1 \leq $i$ \leq t$. Let $m_i=|V(\overrightarrow{C}_i)|$ for every $1 \leq i \leq t$ and let $m_{t+1}=|\Gamma^*|-n$. Since $m_{t+1}\geq 4$, using the sequence $\{m_i\}_{i=1}^{t+1}$, by Theorem~\ref{Tannenbaum1} and Corollary~\ref{nowemain}, we get the result.
\end{proof}

If we consider all Abelian groups, then by Corollary \ref{liniowe} and Lemma~\ref{lemat} one can easily see the following:

\begin{cor}Any digraph\/ $\overrightarrow{G}$ of order\/ $n$ with no weakly connected components
 of order less than\/ $4$ has a $\Gamma$-irregular labeling for every\/ $\Gamma$ such that $|\Gamma|\geq 2n+1$.
\end{cor}

We think that the results from Corollary~\ref{liniowe} can be achieved for all Abelian groups and digraphs with no weakly connected components
 of order less than\/ $3$, therefore we finish this subsection with the following conjecture.

\begin{conj}There exists a constant $K$ such that any digraph\/ $\overrightarrow{G}$ of order\/ $n$ with no weakly connected components
 of order less than\/ $3$ has a $\Gamma$-irregular labeling for every\/ $\Gamma$ such that $|\Gamma|\geq n+K$.
\end{conj}

\subsection{\texorpdfstring{$\Gamma$}{Gamma}-antimagic labeling}\label{subsec:any}

Let $\Gamma$ be an Abelian group. A $\Gamma$-antimagic labeling of a graph $G$ is defined as a bijection 
$f\colon E(G)\to \Gamma$, where the weight of each vertex, which is the sum of the labels on its incident edges, is distinct for all vertices. \citet{KLR} demonstrated that if $\Gamma$ has a unique involution, then no tree of order $|\Gamma|$ has a $\Gamma$-antimagic labeling. They conjectured that a tree of order $|\Gamma|$ has a $\Gamma$-antimagic labeling if and only if $\Gamma$ does not have a unique involution. By applying the same approach used by \citet{KLR}  for 2-trees, and using Corollary~\ref{nowemain}, we derive the following result:

\begin{cor}Every $4$-tree\footnote{A $4$-tree $T$ is a rooted tree, where every vertex that is not a leaf has at least four children.} $T$ of order $|\Gamma|$ admits a $\Gamma$-antimagic labeling if and only if $|I(\Gamma)|\neq 1$.\end{cor}
\begin{proof} 
The necessity of the condition follows directly from the work of~\citet[Theorem E]{KLR}. Thus, let us assume that $|I(\Gamma)|\neq 1$. Let $\{v_i\}_{i=1}^{t}$ be the vertices of $T$ which are not leaves. Let us denote their corresponding numbers of children by $\{m_i\}_{i=1}^{t}$. Thus $\sum_{i=1}^{t}m_i=|\Gamma|-1$. Since $T$ is a
$4$-tree, we have that $m_i \geq 4$ for every  $1 \leq i \leq t$. The group $\Gamma$ has $4$-ZSPP by Corollary~\ref{nowemain}, and there exists a zero-sum partition $\{S_i\}_{i=1}^{t}$ of $\Gamma^*$ such that $|S_i|=m_i$ for $1 \leq i \leq t$. For every $1 \leq i \leq t$, we label the edges of the set $E_i=\{v_iw\colon w$ is a child of $v_i\}$ by the elements of $S_i$. The edges of $T$ are labeled bijectively with the non-zero elements of $\Gamma$, and the sum of
the labels in every $E_i$ is $0$. Since every vertex of $T$, except the root (whose weight is $0$), has a unique parent, it easily follows that the weights are pairwise distinct.
\end{proof}

\subsection{\texorpdfstring{$\Gamma$}{Gamma}-distance-magic and -antimagic labeling in graphs with twins}\label{ssec:twins}

\citet{Fro1} defined the notion of a \textit{group-distance-magic labeling} of a graph $G=(V,E)$. In such a labeling, the vertices of the graph are labeled through a bijection with the elements of an Abelian group $\Gamma$. The weight of each vertex is computed as the sum (in $\Gamma$) of the labels assigned to its neighbors. If all weights are the same, then it is a \textit{$\Gamma$-distance-magic labeling}. If all weights are different, then it is a \textit{$\Gamma$-distance-antimagic labeling} \citep{AnhCicFroSimQiu}.

For ease of presentation of the results in this subsection we will create labelings using only the non-zero elements of $\Gamma$, and speak of \textit{$\Gamma^*$-distance-magic labeling}, etc. It is easy to see that these results can be amended to results on labelings with all elements of $\Gamma$ by dealing with the $0$ element accordingly.

Recall that, in a graph $G=(V,E)$, a set of vertices $M \subset V$ is a \textit{module} if for every pair of vertices $x, y$ with $x,y \in M$ there is $N(x)\setminus M = N(y) \setminus M$. In other words, a module $M$ is a set of vertices that share the same neighbors outside of $M$. Two vertices $x$ and $y$ are called \textit{twins} if $\{x,y\}$ is a module. If additionally $\{x,y\} \in E$, then they are \textit{true twins}, and they are \textit{false twins} otherwise. 

For ease of presentation, we assume that every vertex $x$ is both a false and true twin of itself (in the literature, the relation of being twins is usually considered only for distinct vertices). In this context, notice that the relation of being false twins is an equivalence relation, and the same holds for the relation of being true twins. So they define equivalence classes corresponding to inclusion maximal sets of (false or true) twins. Any set of (false or true) twins is a module.

Based on the above observations, we can create $\Gamma^*$-distance-magic and -antimagic labelings for groups and graphs that satisfy certain conditions detailed in the following propositions.

\begin{prop}\label{prop:twinm}
Let $G=(V,E)$ be a graph with a partition of the vertex set into subsets of false twins $\{V_i\}_{i=1}^t$ for a positive integer $t$. Let $m_i = |V_i|$ for every $1 \leq i \leq t$. Let $\Gamma$ be an Abelian group of order $m$ such that $\sum_{i=1}^t m_i=m-1$. If the non-zero elements of $\Gamma$ can be partitioned into zero-sum subsets $\{S_i\}_{i=1}^t$ such that $|S_i| = m_i$ for every $1 \leq i \leq t$, then $G$ has a $\Gamma^*$-distance-magic labeling.
\end{prop}
\begin{proof}
Let us assign to the vertices of every $V_i$ with $1 \leq i \leq t$ unique elements of the corresponding $S_i$, as in the statement of the proposition. It can be done, since $|V_i| = |S_i|$. Let us show that, under this labeling, the weight of every vertex in $V$ is $0$.

Take any vertex $v \in V_a$ for any $1 \leq a \leq t$. Consider the sum of labels of the neighbors of $v$. Consider any $b\neq a$ with $1\leq b \leq t$. Since $V_b$ is a module, either $v$ is adjacent to no vertices in $V_b$ or to all of them. In both cases, the contribution of $N(v) \cap V_b$ to the weight of $v$ is $0$. 

The only set of neighbors of $v$ left to consider is $N(v)\cap V_a$. $v$ has no neighbors in $V_a$ since $V_a$ is a set of false twins. So the contribution of $N(v)\cap V_a$ to the weight of $v$ is also $0$.
\end{proof}

\begin{prop}\label{prop:twinam}
Let $G=(V,E)$ be a graph with a partition of the vertex set into subsets of true twins $\{V_i\}_{i=1}^t$ for a positive integer $t$. Let $m_i = |V_i|$ for every $1 \leq i \leq t$. Let $\Gamma$ be an Abelian group of order $m$ such that $\sum_{i=1}^t m_i=m-1$. If the non-zero elements of $\Gamma$ can be partitioned into zero-sum subsets $\{S_i\}_{i=1}^t$ such that $|S_i| = m_i$ for every $1 \leq i \leq t$, then $G$ has a $\Gamma^*$-distance-antimagic labeling.
\end{prop}

\begin{proof}
Let us do the labeling as in the proof of Proposition \ref{prop:twinm}. Like therein, take any vertex $v \in V_a$ for any $1 \leq a \leq t$ and consider its weight (the sum of labels of the neighbors of $v$). Let $\gamma$ be the element of $\Gamma$ assigned to $v$.

Like in the proof of Proposition \ref{prop:twinm}, for any $b\neq a$ with $1\leq b \leq t$, the contribution of $N(v) \cap V_b$ to the weight of $v$ is $0$. But now $V_a$ is a clique. The sum of the elements of $\Gamma$ assigned to the vertices in $V_a$ is $0$. So the sum of the labels assigned to $N(v) \cap V_a$ is $-\gamma$, the element of $\Gamma$ inverse to $\gamma$.

Since the assignment of non-zero elements of $\Gamma$ to the elements of $V$ is a bijection, by uniqueness of inverses, the sums of labels of neighbors of the vertices in $V$ are unique too.
\end{proof}

The following corollaries are direct consequences of Propositions \ref{prop:twinm} and \ref{prop:twinam}, Corollary \ref{nowemain}, and the properties of twins mentioned above. Note that we say that a vertex has $k$ twins if it has $k$ twins different from itself. 

\begin{cor}\label{cor_twinm}
Let $G=(V,E)$ be a graph in which every vertex has at least $3$ false twins. Let $\Gamma$ be a finite Abelian group of order $|V|+1$ such that $|I(\Gamma)|\neq 1$. Then $G$ has a $\Gamma^*$-distance-magic labeling.
\end{cor}

Notice that Corollary \ref{cor_twinm} extends the results on $\Gamma$-distance-magic labelings of complete $k$-partite graphs presented by \citet{CicrSuch}.

\begin{cor}\label{cor_twinam}
Let $G=(V,E)$ be a graph in which every vertex has at least $3$ true twins. Let $\Gamma$ be a finite Abelian group of order $|V|+1$ such that $|I(\Gamma)|\neq 1$. Then $G$ has a $\Gamma^*$-distance-antimagic labeling.
\end{cor}

For the last two results of this subsection, let us recall the definition of a $G$-join introduced by~\citet{Sabidussi61}.

\begin{dfn}[\citep{Sabidussi61}]
Let $G=(V,E)$ be a graph and $\{X_v\}_{v\in V}$ a collection of graphs indexed by $V$. The $G$-join of $\{X_v\}_{v\in V}$ is the graph $H$ given by
$V(H) = \{(x,v) \mid x \in X_v, v\in V\}$, 
$E(H)=\{\{(x,v),(x',v')\} \mid \{v,v'\} \in E \text{, or } v=v'  \text{ and } \{x,x'\} \in E(X_v) \}$.
\end{dfn}

Notice that, given a graph $G=(V,E)$, in the graph $H$ that is a $G$-join of $\{X_v\}_{v\in V}$, the set of vertices $\{(x,v) \mid x \in X_v\}$ for every $v\in V$ is a module in $H$. Moreover, if $X_v$ is an empty graph, then the vertices of $H$ in $\{(x,v) \mid x \in X_v\}$ are false twins. If $X_v$ is a complete graph, then the vertices of $H$ in $\{(x,v) \mid x \in X_v\}$ are true twins. 

Recall that $X$ is an \textit{empty graph} if $E(X)=\emptyset$. In other words, its vertex set is stable. We use $K_k$ to denote the complete graph on $k$ vertices and $\overline{K_k}$ for the empty graph on $k$ vertices. The following corollaries are direct consequences of Corollaries \ref{cor_twinm} and \ref{cor_twinam}.



\begin{cor}\label{cor_blowuptwinm}
Let $G=(V,E)$ be any graph. Then there exists a $G$-join $H$ of a sequence of empty graphs of appropriate orders such that every vertex in $H$ has at least $3$ false twins. In particular, $H$ has a $\Gamma^*$-distance-magic labeling for every finite Abelian group $\Gamma$ with $|\Gamma|=|V(H)|+1$ and $|I(\Gamma)|\neq 1$.
\end{cor}


\begin{cor}\label{cor_blowuptwinam}
Let $G=(V,E)$ be any graph. Then there exists a $G$-join $H$ of a sequence of complete graphs of appropriate orders such that every vertex in $H$ has at least $3$ true twins. In particular, $H$ has a $\Gamma^*$-distance-antimagic labeling for every finite Abelian group $\Gamma$ with $|\Gamma|=|V(H)|+1$ and $|I(\Gamma)|\neq 1$.
\end{cor}


Note that the equivalence classes of twins in a graph can be found in linear time~\citep{HABIB201041}. So, if a respective zero-sum partition of $\Gamma^*$ is given or it can be found efficiently, then we can efficiently compute the corresponding labeling.

{
Note that, based on Corollary \ref{cor_m}, the results presented in this subsection can be made stronger in particular cases. If the order $n$ of the graph satisfies $n+1 \geq 4$ and $(n+1) \bmod 4 \neq 2$, then it is enough to only require at least one twin for every vertex. If additionally $(n+1) \bmod 4 \in \{1,3\}$, then any Abelian group of order $n+1$ can be used for the labeling. If $(n+1) \bmod 4 = 0$, then there exists a particular group that can be used. In the case of $(n+1) \bmod 4 = 2$, no group with a suitable zero-sum partition exists, but the graph can be ``fixed'' by adding an additional twin for one of the vertices, and thus reducing to the case $(n+1) \bmod 4 = 3$.
}

\section{Final Remarks}\label{sec:final}
On the one hand we showed that Conjecture~\ref{conjectureT} is not true in general, but, on the other hand, we have achieved important progress towards proving Conjecture~\ref{conjecture}. Recall that any group $\Gamma$ can be factorized as $\Gamma\cong L\times H$, where $L$ is the Sylow $2$-group of $\Gamma$ and the order of $H$ is odd. In this context, towards proving Conjecture~\ref{conjecture}, we leave open only the question if $\Gamma$ has not only $4$-ZSPP, but also $3$-ZSPP, in the case where $(|H| \bmod{6}) = 5$.

Finally, let us complement Definition \ref{dfn:skolem}. Let $S$ be a subset of cardinality $m=6k+s$, for a positive integer $k$ and $s\in\{0,2,4\}$, of a finite Abelian group $\Gamma$. A partition of $S$ into $k$ good $6$-subsets and $s/2$ zero-sum $2$-subsets is called a {\em Skolem partition of $S$}. It is known that $\zet_m^*\setminus I(\zet_m)${, for $m=6k+s$ with any positive integer $k$,} has a Skolem partition when $s\in\{0,4\}$, but the situation is slightly different for $s=2$. For $m\equiv 2$ or $8\pmod{24}$, there exists a Skolem partition, whereas for $m\equiv 14$ or $20\pmod{24}$, such a partition does not exist~\citep[Lemma 4]{Tannenbaum1}). Note that there exist groups $\Gamma$ with $|I(\Gamma)|>1$ such that the set $R =\Gamma^* \setminus I(\Gamma)$ has a Skolem partition. For instance, let $\Gamma\cong(\zet_2)^{\eta}\times H$  for some natural number $\eta>1$ and $|H|\equiv 1\pmod 6$. By Theorem~\ref{Tannenbaum1}, there exists a partition of $H^*$ into good $6$-subsets:
$$H^* = \bigcup\limits_{i=1}^{|H^*|/6}\{b_i, c_i,-b_i - c_i,-b_i,-c_i, b_i + c_i\}.$$

Let $L=(\zet_2)^{\eta}$. Note that $I(\Gamma)\cong L^*$. By  Lemma~\ref{bijection}, there exist $\phi,\varphi\in$Bij$(L)$ such that $a+\phi(a)+\varphi(a)=0$  for every $a\in L$. Thus, 
$$R = \bigcup \limits_{a\in L}\bigcup\limits_{i=1}^{|H^*|/6}\{(a,b_i),(\phi(a), c_i),(\varphi(a),-b_i-c_i),(-a,-b_i),(-\phi(a),-c_i),(-\varphi(a), b_i+c_i)\}.$$
We intuit these are not the only groups for which the set $R$ has a Skolem partition. This motivates the following problem.  
\begin{prob}\label{conjectureSK}Characterize finite Abelian groups $\Gamma$ for which $R =\Gamma^* \setminus I(\Gamma)$ has a Skolem partition.
\end{prob}

{
As reflected in Proposition~\ref{prop:twinm}, given a graph $G=(V,E)$ with a partition of the vertex set into subsets of false twins,
existence of a zero-sum partition of an Abelian group $\Gamma$ into subsets of corresponding cardinalities is a sufficient condition for existence of a $\Gamma^*$-distance-magic labeling of $G$. One can easily see that this condition is not necessary. For example, consider the group $\zet_6$. There exists no zero-sum partition of $\zet_6^*$ as required in Proposition~\ref{prop:twinm}. On the other hand, $\{\{1,2\}, \{3\}, \{4,5\} \}$ is a constant-sum partition of $\zet^*_6$, with constant sum $3$, which leads to a $\zet_6^*$-distance-magic labeling
of the complete tripartite graph $K_{1,2,2}$. So far, constant-sum partitions of Abelian groups have been studied only for the case of partitioning into three subsets~\citep{Cic3}. It would be interesting to consider constant-sum partitions into more than three subsets.
}

{On the algorithmic side, given a finite Abelian group $\Gamma$ of order $m$ and an integer partition $\{m_i\}_{i=1}^t$ of $m-1$, it would be interesting to know the time complexity (and an efficient algorithm) of finding a zero-sum partition that realizes $\{m_i\}_{i=1}^t$ (if it exists).}
\section{Acknowledgments}

The authors would like to thank the anonymous reviewers for their careful reading of the manuscript and their insightful comments and suggestions.
\section{Statements and Declarations}
\begin{itemize}
\item Funding. This work was supported by the program ``Excellence initiative – research university'' for the AGH University.
\item Competing interests. The authors have no relevant financial or non-financial interests to disclose.
\item Availability of data and materials. All data generated or analyzed during this study are included in this published article. 
\end{itemize}

\nocite{*}
\bibliographystyle{abbrvnat}
\bibliography{grupy}

\newpage
\begin{appendices}
Throughout the annexes, the notation ``\verb|a*3  b*4  c*5|'' refers to $a$ subsets of size $3$, $b$ subsets of size $4$, and $c$ subsets of size $5$. In each subset partition, we present one set per line. Each tuple represents one element of the partitioned set, with the elements of the tuple following the order of the direct product.

\section{Zero-sum partitions of 
\texorpdfstring{$\left((\zet_2)^3\times\zet_3\right)^*$}
{((Z\_2)\textasciicircum 3 x Z\_3)*}
with 
\texorpdfstring{$b\leq2$}
{b<=2}
}\label{a:2223}

\footnotesize
\begin{verbatim}
[2, 2, 2, 3]
[5, 2, 0]
[
[[0, 0, 1, 0], [0, 1, 0, 0], [0, 1, 1, 0]],
[[1, 0, 0, 0], [0, 0, 0, 1], [1, 0, 0, 2]],
[[1, 1, 1, 0], [0, 0, 1, 1], [1, 1, 0, 2]],
[[0, 1, 0, 1], [1, 0, 0, 1], [1, 1, 0, 1]],
[[0, 0, 1, 2], [0, 1, 0, 2], [0, 1, 1, 2]],
[[1, 0, 1, 1], [1, 1, 1, 1], [1, 0, 1, 2], [1, 1, 1, 2]],
[[1, 0, 1, 0], [1, 1, 0, 0], [0, 1, 1, 1], [0, 0, 0, 2]]
]
A partition for subsets of sizes:  5*3  2*4  0*5

[2, 2, 2, 3]
[6, 0, 1]
[
[[0, 0, 1, 0], [0, 1, 0, 0], [0, 1, 1, 0]],
[[1, 0, 0, 0], [0, 0, 0, 1], [1, 0, 0, 2]],
[[1, 1, 1, 0], [0, 0, 1, 1], [1, 1, 0, 2]],
[[0, 1, 0, 1], [1, 0, 0, 1], [1, 1, 0, 1]],
[[0, 0, 1, 2], [0, 1, 0, 2], [0, 1, 1, 2]],
[[1, 1, 0, 0], [0, 1, 1, 1], [1, 0, 1, 2]],
[[1, 0, 1, 0], [1, 0, 1, 1], [1, 1, 1, 1], [0, 0, 0, 2], [1, 1, 1, 2]]
]
A partition for subsets of sizes:  6*3  0*4  1*5

[2, 2, 2, 3]
[3, 1, 2]
[
[[0, 0, 1, 0], [0, 1, 0, 0], [0, 1, 1, 0]],
[[1, 0, 0, 0], [0, 0, 0, 1], [1, 0, 0, 2]],
[[1, 1, 1, 0], [0, 0, 1, 1], [1, 1, 0, 2]],
[[0, 1, 0, 1], [0, 1, 1, 1], [0, 0, 0, 2], [0, 0, 1, 2]],
[[1, 1, 0, 0], [1, 0, 0, 1], [1, 1, 1, 1], [0, 1, 0, 2], [1, 1, 1, 2]],
[[1, 0, 1, 0], [1, 0, 1, 1], [1, 1, 0, 1], [0, 1, 1, 2], [1, 0, 1, 2]]
]
A partition for subsets of sizes:  3*3  1*4  2*5

[2, 2, 2, 3]
[0, 2, 3]
[
[[0, 0, 1, 0], [0, 1, 0, 0], [1, 0, 0, 0], [1, 1, 1, 0]],
[[1, 0, 1, 0], [1, 1, 0, 0], [0, 0, 0, 1], [0, 1, 1, 2]],
[[0, 0, 1, 1], [0, 1, 0, 1], [0, 1, 1, 1], [1, 0, 0, 1], [1, 0, 0, 2]],
[[1, 1, 0, 1], [0, 0, 0, 2], [0, 0, 1, 2], [0, 1, 0, 2], [1, 0, 1, 2]],
[[0, 1, 1, 0], [1, 0, 1, 1], [1, 1, 1, 1], [1, 1, 0, 2], [1, 1, 1, 2]]
]
A partition for subsets of sizes:  0*3  2*4  3*5
\end{verbatim}

\newpage

\begin{verbatim}
[2, 2, 2, 3]
[1, 0, 4]
[
[[0, 0, 1, 0], [0, 1, 0, 0], [0, 1, 1, 0]],
[[1, 0, 0, 0], [1, 0, 1, 0], [1, 1, 0, 0], [0, 0, 0, 1], [1, 1, 1, 2]],
[[0, 0, 1, 1], [0, 1, 0, 1], [0, 1, 1, 1], [1, 0, 0, 1], [1, 0, 0, 2]],
[[1, 1, 0, 1], [0, 0, 0, 2], [0, 0, 1, 2], [0, 1, 0, 2], [1, 0, 1, 2]],
[[1, 1, 1, 0], [1, 0, 1, 1], [1, 1, 1, 1], [0, 1, 1, 2], [1, 1, 0, 2]]
]
A partition for subsets of sizes:  1*3  0*4  4*5
\end{verbatim}

\normalsize
\section{Zero-sum partitions of
\texorpdfstring{$\left((\zet_2)^3\times\zet_5\right)^*$}
{((Z\_2)\textasciicircum 3 x Z\_5)*}
 with 
\texorpdfstring{$b\leq4$}
{b<=4}
}\label{a:2225}
\footnotesize

\begin{verbatim}
[2, 2, 2, 5]
[13, 0, 0]
[
[[0, 0, 1, 0], [0, 1, 0, 0], [0, 1, 1, 0]],
[[1, 0, 0, 0], [0, 0, 0, 1], [1, 0, 0, 4]],
[[1, 1, 1, 0], [0, 0, 1, 1], [1, 1, 0, 4]],
[[0, 1, 0, 1], [0, 1, 1, 1], [0, 0, 1, 3]],
[[1, 0, 1, 1], [1, 1, 0, 1], [0, 1, 1, 3]],
[[0, 0, 1, 2], [0, 0, 0, 4], [0, 0, 1, 4]],
[[1, 0, 0, 2], [0, 1, 1, 4], [1, 1, 1, 4]],
[[1, 0, 0, 1], [0, 1, 0, 2], [1, 1, 0, 2]],
[[0, 1, 0, 3], [1, 1, 1, 3], [1, 0, 1, 4]],
[[1, 0, 0, 3], [1, 1, 0, 3], [0, 1, 0, 4]],
[[1, 1, 0, 0], [0, 1, 1, 2], [1, 0, 1, 3]],
[[1, 1, 1, 1], [0, 0, 0, 2], [1, 1, 1, 2]],
[[1, 0, 1, 0], [1, 0, 1, 2], [0, 0, 0, 3]]
]
A partition for subsets of sizes: 13*3  0*4  0*5

[2, 2, 2, 5]
[9, 3, 0]
[
[[0, 0, 1, 0], [0, 1, 0, 0], [0, 1, 1, 0]],
[[1, 0, 0, 0], [0, 0, 0, 1], [1, 0, 0, 4]],
[[1, 1, 1, 0], [0, 0, 1, 1], [1, 1, 0, 4]],
[[0, 1, 0, 1], [0, 1, 1, 1], [0, 0, 1, 3]],
[[1, 0, 1, 1], [1, 1, 0, 1], [0, 1, 1, 3]],
[[0, 0, 1, 2], [0, 0, 0, 4], [0, 0, 1, 4]],
[[1, 0, 0, 2], [0, 1, 1, 4], [1, 1, 1, 4]],
[[1, 1, 1, 2], [0, 1, 0, 4], [1, 0, 1, 4]],
[[1, 0, 0, 1], [0, 1, 0, 2], [1, 1, 0, 2]],
[[0, 0, 0, 2], [0, 1, 1, 2], [1, 0, 0, 3], [1, 1, 1, 3]],
[[1, 1, 1, 1], [0, 0, 0, 3], [0, 1, 0, 3], [1, 0, 1, 3]],
[[1, 0, 1, 0], [1, 1, 0, 0], [1, 0, 1, 2], [1, 1, 0, 3]]
]
A partition for subsets of sizes:  9*3  3*4  0*5
\end{verbatim}

\newpage

\begin{verbatim}
[2, 2, 2, 5]
[10, 1, 1]
[
[[0, 0, 1, 0], [0, 1, 0, 0], [0, 1, 1, 0]],
[[1, 0, 0, 0], [0, 0, 0, 1], [1, 0, 0, 4]],
[[1, 1, 1, 0], [0, 0, 1, 1], [1, 1, 0, 4]],
[[0, 1, 0, 1], [0, 1, 1, 1], [0, 0, 1, 3]],
[[1, 0, 1, 1], [1, 1, 0, 1], [0, 1, 1, 3]],
[[0, 0, 1, 2], [0, 0, 0, 4], [0, 0, 1, 4]],
[[1, 0, 0, 2], [0, 1, 1, 4], [1, 1, 1, 4]],
[[1, 1, 1, 2], [0, 1, 0, 4], [1, 0, 1, 4]],
[[1, 0, 0, 1], [0, 1, 0, 2], [1, 1, 0, 2]],
[[1, 1, 0, 0], [0, 0, 0, 2], [1, 1, 0, 3]],
[[1, 1, 1, 1], [0, 0, 0, 3], [0, 1, 0, 3], [1, 0, 1, 3]],
[[1, 0, 1, 0], [0, 1, 1, 2], [1, 0, 1, 2], [1, 0, 0, 3], [1, 1, 1, 3]]
]
A partition for subsets of sizes: 10*3  1*4  1*5

[2, 2, 2, 5]
[6, 4, 1]
[
[[0, 0, 1, 0], [0, 1, 0, 0], [0, 1, 1, 0]],
[[1, 0, 0, 0], [0, 0, 0, 1], [1, 0, 0, 4]],
[[1, 1, 1, 0], [0, 0, 1, 1], [1, 1, 0, 4]],
[[0, 1, 0, 1], [0, 1, 1, 1], [0, 0, 1, 3]],
[[1, 0, 1, 1], [1, 1, 0, 1], [0, 1, 1, 3]],
[[0, 0, 1, 2], [0, 0, 0, 4], [0, 0, 1, 4]],
[[0, 1, 1, 2], [1, 0, 0, 2], [1, 0, 1, 2], [0, 1, 0, 4]],
[[1, 0, 0, 1], [0, 0, 0, 3], [0, 1, 0, 3], [1, 1, 0, 3]],
[[1, 1, 1, 1], [1, 1, 1, 2], [1, 0, 1, 3], [1, 0, 1, 4]],
[[1, 1, 0, 0], [0, 1, 0, 2], [0, 1, 1, 4], [1, 1, 1, 4]],
[[1, 0, 1, 0], [0, 0, 0, 2], [1, 1, 0, 2], [1, 0, 0, 3], [1, 1, 1, 3]]
]
A partition for subsets of sizes:  6*3  4*4  1*5

[2, 2, 2, 5]
[7, 2, 2]
[
[[0, 0, 1, 0], [0, 1, 0, 0], [0, 1, 1, 0]],
[[1, 0, 0, 0], [0, 0, 0, 1], [1, 0, 0, 4]],
[[1, 1, 1, 0], [0, 0, 1, 1], [1, 1, 0, 4]],
[[0, 1, 0, 1], [0, 1, 1, 1], [0, 0, 1, 3]],
[[1, 0, 1, 1], [1, 1, 0, 1], [0, 1, 1, 3]],
[[0, 0, 1, 2], [0, 0, 0, 4], [0, 0, 1, 4]],
[[1, 0, 0, 2], [0, 1, 1, 4], [1, 1, 1, 4]],
[[1, 1, 0, 2], [1, 1, 1, 2], [1, 0, 0, 3], [1, 0, 1, 3]],
[[1, 1, 1, 1], [0, 0, 0, 2], [0, 1, 0, 3], [1, 0, 1, 4]],
[[0, 1, 1, 2], [0, 0, 0, 3], [1, 1, 0, 3], [1, 1, 1, 3], [0, 1, 0, 4]],
[[1, 0, 1, 0], [1, 1, 0, 0], [1, 0, 0, 1], [0, 1, 0, 2], [1, 0, 1, 2]]
]
A partition for subsets of sizes:  7*3  2*4  2*5
\end{verbatim}

\newpage

\begin{verbatim}
[2, 2, 2, 5]
[8, 0, 3]
[
[[0, 0, 1, 0], [0, 1, 0, 0], [0, 1, 1, 0]],
[[1, 0, 0, 0], [0, 0, 0, 1], [1, 0, 0, 4]],
[[1, 1, 1, 0], [0, 0, 1, 1], [1, 1, 0, 4]],
[[0, 1, 0, 1], [0, 1, 1, 1], [0, 0, 1, 3]],
[[1, 0, 1, 1], [1, 1, 0, 1], [0, 1, 1, 3]],
[[0, 0, 1, 2], [0, 0, 0, 4], [0, 0, 1, 4]],
[[1, 0, 0, 2], [0, 1, 1, 4], [1, 1, 1, 4]],
[[1, 1, 1, 2], [0, 1, 0, 4], [1, 0, 1, 4]],
[[1, 0, 0, 1], [1, 1, 1, 1], [0, 0, 0, 2], [1, 0, 0, 3], [1, 1, 1, 3]],
[[1, 1, 0, 0], [0, 1, 0, 2], [1, 1, 0, 2], [0, 0, 0, 3], [0, 1, 0, 3]],
[[1, 0, 1, 0], [0, 1, 1, 2], [1, 0, 1, 2], [1, 0, 1, 3], [1, 1, 0, 3]]
]
A partition for subsets of sizes:  8*3  0*4  3*5

[2, 2, 2, 5]
[4, 3, 3]
[
[[0, 0, 1, 0], [0, 1, 0, 0], [0, 1, 1, 0]],
[[1, 0, 0, 0], [0, 0, 0, 1], [1, 0, 0, 4]],
[[1, 1, 1, 0], [0, 0, 1, 1], [1, 1, 0, 4]],
[[0, 1, 0, 1], [0, 1, 1, 1], [0, 0, 1, 3]],
[[1, 0, 1, 1], [1, 1, 0, 1], [1, 1, 1, 1], [1, 0, 0, 2]],
[[0, 0, 1, 2], [0, 1, 0, 2], [0, 1, 1, 2], [0, 0, 0, 4]],
[[1, 0, 1, 2], [1, 1, 0, 2], [0, 0, 0, 3], [0, 1, 1, 3]],
[[1, 0, 0, 1], [0, 1, 0, 3], [1, 0, 0, 3], [0, 0, 1, 4], [0, 1, 1, 4]],
[[0, 0, 0, 2], [1, 1, 1, 2], [1, 0, 1, 3], [1, 0, 1, 4], [1, 1, 1, 4]],
[[1, 0, 1, 0], [1, 1, 0, 0], [1, 1, 0, 3], [1, 1, 1, 3], [0, 1, 0, 4]]
]
A partition for subsets of sizes:  4*3  3*4  3*5

[2, 2, 2, 5]
[5, 1, 4]
[
[[0, 0, 1, 0], [0, 1, 0, 0], [0, 1, 1, 0]],
[[1, 0, 0, 0], [0, 0, 0, 1], [1, 0, 0, 4]],
[[1, 1, 1, 0], [0, 0, 1, 1], [1, 1, 0, 4]],
[[0, 1, 0, 1], [0, 1, 1, 1], [0, 0, 1, 3]],
[[1, 0, 1, 1], [1, 1, 0, 1], [0, 1, 1, 3]],
[[0, 0, 0, 2], [0, 0, 1, 2], [0, 1, 0, 2], [0, 1, 1, 4]],
[[1, 0, 0, 2], [1, 0, 1, 2], [0, 0, 0, 3], [0, 0, 0, 4], [0, 0, 1, 4]],
[[1, 0, 0, 1], [0, 1, 0, 3], [1, 0, 0, 3], [1, 0, 1, 4], [1, 1, 1, 4]],
[[1, 1, 0, 2], [1, 0, 1, 3], [1, 1, 0, 3], [1, 1, 1, 3], [0, 1, 0, 4]],
[[1, 0, 1, 0], [1, 1, 0, 0], [1, 1, 1, 1], [0, 1, 1, 2], [1, 1, 1, 2]]
]
A partition for subsets of sizes:  5*3  1*4  4*5
\end{verbatim}

\newpage

\begin{verbatim}
[2, 2, 2, 5]
[1, 4, 4]
[
[[0, 0, 1, 0], [0, 1, 0, 0], [0, 1, 1, 0]],
[[1, 0, 0, 0], [1, 0, 1, 0], [1, 1, 0, 0], [1, 1, 1, 0]],
[[0, 0, 0, 1], [0, 0, 1, 1], [0, 1, 0, 1], [0, 1, 1, 2]],
[[1, 0, 0, 1], [1, 0, 1, 1], [1, 1, 0, 1], [1, 1, 1, 2]],
[[0, 0, 0, 2], [0, 0, 1, 2], [0, 1, 0, 2], [0, 1, 1, 4]],
[[1, 0, 0, 2], [1, 0, 1, 2], [0, 0, 0, 3], [0, 0, 0, 4], [0, 0, 1, 4]],
[[1, 1, 1, 1], [0, 1, 0, 3], [0, 1, 1, 3], [0, 1, 0, 4], [1, 0, 0, 4]],
[[1, 1, 0, 2], [0, 0, 1, 3], [1, 1, 0, 3], [1, 1, 1, 3], [1, 1, 0, 4]],
[[0, 1, 1, 1], [1, 0, 0, 3], [1, 0, 1, 3], [1, 0, 1, 4], [1, 1, 1, 4]]
]
A partition for subsets of sizes:  1*3  4*4  4*5

[2, 2, 2, 5]
[2, 2, 5]
[
[[0, 0, 1, 0], [0, 1, 0, 0], [0, 1, 1, 0]],
[[1, 0, 0, 0], [0, 0, 0, 1], [1, 0, 0, 4]],
[[1, 1, 1, 0], [0, 0, 1, 1], [0, 1, 0, 1], [1, 0, 0, 3]],
[[0, 1, 1, 1], [1, 0, 0, 1], [1, 0, 1, 1], [0, 1, 0, 2]],
[[1, 1, 1, 1], [0, 0, 0, 2], [0, 0, 1, 2], [0, 1, 1, 2], [1, 0, 1, 3]],
[[1, 0, 0, 2], [1, 0, 1, 2], [0, 0, 0, 3], [0, 0, 0, 4], [0, 0, 1, 4]],
[[1, 0, 1, 0], [0, 1, 0, 3], [0, 1, 0, 4], [0, 1, 1, 4], [1, 1, 0, 4]],
[[1, 1, 1, 2], [0, 1, 1, 3], [1, 1, 0, 3], [1, 1, 1, 3], [1, 0, 1, 4]],
[[1, 1, 0, 0], [1, 1, 0, 1], [1, 1, 0, 2], [0, 0, 1, 3], [1, 1, 1, 4]]
]
A partition for subsets of sizes:  2*3  2*4  5*5

[2, 2, 2, 5]
[3, 0, 6]
[
[[0, 0, 1, 0], [0, 1, 0, 0], [0, 1, 1, 0]],
[[1, 0, 0, 0], [0, 0, 0, 1], [1, 0, 0, 4]],
[[1, 1, 1, 0], [0, 0, 1, 1], [1, 1, 0, 4]],
[[0, 1, 0, 1], [0, 1, 1, 1], [1, 0, 0, 1], [0, 0, 0, 3], [1, 0, 1, 4]],
[[1, 1, 1, 1], [0, 0, 0, 2], [0, 0, 1, 2], [0, 1, 0, 2], [1, 0, 0, 3]],
[[1, 0, 0, 2], [1, 0, 1, 2], [0, 1, 0, 3], [0, 0, 0, 4], [0, 1, 1, 4]],
[[1, 1, 0, 2], [0, 0, 1, 3], [0, 1, 1, 3], [1, 0, 1, 3], [0, 0, 1, 4]],
[[1, 0, 1, 1], [1, 1, 0, 1], [0, 1, 1, 2], [1, 1, 1, 2], [1, 1, 1, 4]],
[[1, 0, 1, 0], [1, 1, 0, 0], [1, 1, 0, 3], [1, 1, 1, 3], [0, 1, 0, 4]]
]
A partition for subsets of sizes:  3*3  0*4  6*5
\end{verbatim}

\newpage

\begin{verbatim}
[2, 2, 2, 5]
[0, 1, 7]
[
[[0, 0, 1, 0], [0, 1, 0, 0], [1, 0, 0, 0], [1, 1, 1, 0]],
[[1, 0, 1, 0], [1, 1, 0, 0], [0, 0, 0, 1], [0, 0, 1, 1], [0, 1, 0, 3]],
[[0, 1, 0, 1], [0, 1, 1, 1], [1, 0, 0, 1], [0, 0, 0, 3], [1, 0, 1, 4]],
[[1, 1, 1, 1], [0, 0, 0, 2], [0, 0, 1, 2], [0, 1, 0, 2], [1, 0, 0, 3]],
[[1, 0, 0, 2], [1, 0, 1, 2], [0, 1, 1, 3], [0, 0, 0, 4], [0, 1, 0, 4]],
[[0, 1, 1, 0], [0, 0, 1, 3], [0, 0, 1, 4], [1, 0, 0, 4], [1, 1, 1, 4]],
[[1, 1, 1, 2], [1, 0, 1, 3], [1, 1, 0, 3], [1, 1, 1, 3], [0, 1, 1, 4]],
[[1, 0, 1, 1], [1, 1, 0, 1], [0, 1, 1, 2], [1, 1, 0, 2], [1, 1, 0, 4]]
]
A partition for subsets of sizes:  0*3  1*4  7*5
\end{verbatim}

\normalsize
\section{Zero-sum partitions of 
\texorpdfstring{$\left((\zet_2)^3\times\zet_7\right)^*$}
{((Z\_2) \textasciicircum 3 x Z\_7)*}
 - 2 special cases}\label{a:2227}
\footnotesize

\begin{verbatim}
[2, 2, 2, 7]
[10, 0, 5]
[
[[0, 0, 0, 1], [0, 0, 1, 0], [0, 0, 1, 1]],
[[0, 1, 0, 0], [1, 0, 0, 0], [6, 1, 0, 0]],
[[0, 1, 1, 1], [1, 0, 0, 1], [6, 1, 1, 0]],
[[1, 0, 1, 0], [1, 0, 1, 1], [5, 0, 0, 1]],
[[1, 1, 0, 1], [1, 1, 1, 0], [5, 0, 1, 1]],
[[2, 0, 0, 0], [2, 0, 0, 1], [3, 0, 0, 1]],
[[2, 0, 1, 1], [2, 1, 0, 0], [3, 1, 1, 1]],
[[2, 1, 1, 0], [6, 0, 0, 1], [6, 1, 1, 1]],
[[2, 0, 1, 0], [6, 0, 0, 0], [6, 0, 1, 0]],
[[3, 1, 0, 0], [5, 1, 1, 1], [6, 0, 1, 1]],
[[2, 1, 0, 1], [4, 0, 0, 0], [4, 0, 1, 0], [5, 0, 1, 0], [6, 1, 0, 1]],
[[4, 1, 0, 0], [4, 1, 0, 1], [4, 1, 1, 0], [4, 1, 1, 1], [5, 0, 0, 0]],
[[1, 1, 0, 0], [1, 1, 1, 1], [3, 1, 0, 1], [4, 0, 1, 1], [5, 1, 0, 1]],
[[2, 1, 1, 1], [3, 0, 0, 0], [3, 0, 1, 0], [3, 0, 1, 1], [3, 1, 1, 0]],
[[0, 1, 0, 1], [0, 1, 1, 0], [4, 0, 0, 1], [5, 1, 0, 0], [5, 1, 1, 0]
]
A partition for subsets of sizes: 10*3  0*4  5*5
\end{verbatim}

\begin{verbatim}
[2, 2, 2, 7]
[5, 0, 8]
[
[[0, 0, 0, 1], [0, 0, 1, 0], [0, 0, 1, 1]],
[[0, 1, 0, 0], [1, 0, 0, 0], [6, 1, 0, 0]],
[[0, 1, 1, 1], [1, 0, 0, 1], [6, 1, 1, 0]],
[[1, 0, 1, 0], [1, 0, 1, 1], [5, 0, 0, 1]],
[[1, 1, 0, 1], [1, 1, 1, 0], [5, 0, 1, 1]],
[[2, 0, 0, 0], [2, 0, 0, 1], [2, 0, 1, 0], [2, 0, 1, 1], [6, 0, 0, 0]],
[[2, 1, 0, 1], [2, 1, 1, 0], [2, 1, 1, 1], [3, 0, 0, 0], [5, 1, 0, 0]],
[[3, 0, 1, 0], [3, 0, 1, 1], [3, 1, 0, 0], [6, 0, 1, 0], [6, 1, 1, 1]],
[[3, 1, 1, 1], [4, 0, 0, 0], [4, 0, 0, 1], [4, 0, 1, 1], [6, 1, 0, 1]],
[[4, 1, 0, 0], [4, 1, 0, 1], [4, 1, 1, 0], [4, 1, 1, 1], [5, 0, 0, 0]],
[[1, 1, 0, 0], [1, 1, 1, 1], [2, 1, 0, 0], [5, 0, 1, 0], [5, 1, 0, 1]],
[[3, 0, 0, 1], [3, 1, 0, 1], [4, 0, 1, 0], [5, 1, 1, 1], [6, 0, 0, 1]],
[[0, 1, 0, 1], [0, 1, 1, 0], [3, 1, 1, 0], [5, 1, 1, 0], [6, 0, 1, 1]
]
A partition for subsets of sizes:  5*3  0*4  8*5
\end{verbatim}

\normalsize
\section{Program to check a subset partition}\label{a:program}
We offer a simple program in Python 3 that allows to check easily if a subset partition is a zero-sum partition. It can be executed in a terminal (or in an online Python environment like https://www.online-python.com/ or https://trinket.io/python), with the three elements of the input: description of the group, sizes of the subsets in the partition, and the partition itself, copied-pasted from the annexes.
\footnotesize

\begin{verbatim}
import json

group = json.loads(input())
sizes = json.loads(input())
sets = json.loads(input())
ok = True
for set in sets:
    sums = [0 for pos in range(len(group))]
    for elem in set:
        for pos in range(len(group)):
            sums[pos] = (sums[pos] + elem[pos]) % group[pos]
    if not sums == [0 for i in range(len(group))]:
        ok = False
        break
if ok:
    print("Zero-sum partition")
else:
    print("Not a zero-sum partition")
\end{verbatim}

\end{appendices}

\end{document}